\documentclass[10pt,a4paper]{article}
\usepackage[utf8]{inputenc}
\usepackage[american]{babel}
\usepackage[all]{xy}
\usepackage{enumitem}
\usepackage[dvipsnames]{xcolor}
\usepackage{graphicx}
\usepackage{stmaryrd}
\usepackage{amsfonts,amstext,amssymb,amsmath,amsthm,pspicture,bigints}
\usepackage{fullpage}
\usepackage{moreverb}
\usepackage{pifont}
\usepackage{pst-all}
\usepackage[left=2cm,right=2cm,top=2cm,bottom=2cm]{geometry}
\usepackage{esint}
\usepackage{fourier-orns}
\usepackage{mathrsfs}
\usepackage{array}

\newcommand{\ii }{{\rm i} }
\usepackage{hyperref}
\usepackage{caption}
\usepackage{subcaption}
\usepackage{authblk}
\newcolumntype{C}[1]{>{\centering\arraybackslash}b{#1}}
\newcolumntype{R}[1]{>{\raggedleft\arraybackslash}b{#1}}
\newcolumntype{L}[1]{>{\raggedright\arraybackslash}b{#1}}
\newcolumntype{M}[1]{>{\centering}m{#1}}
\newtheorem{theo}{Theorem}[section]
\newtheorem{defin}{Definition}[section]
\newtheorem{lem}{Lemma}[section]
\newtheorem{prop}{Proposition}[section]

\newtheorem{remark}{Remark}[section]

\usepackage{bbm}
\numberwithin{equation}{section}

\date{}
\title{Vortex rigid motion in quasi-geostrophic shallow-water equations}
\author{Emeric Roulley\thanks{Univ Rennes, CNRS, IRMAR – UMR 6625, F-35000 Rennes, France\\
E-mail address : emeric.roulley@univ-rennes1.fr}}
\begin{document}
\maketitle
\begin{abstract}
In this paper, we prove the existence of relative equilibria with holes for quasi-geostrophic shallow-water equations. More precisely, using bifurcation techniques,  we establish for any $\mathbf{m}$ large enough  the existence of two branches of $\mathbf{m}$-fold doubly-connected V-states bifurcating from any annulus of arbitrary size.
\end{abstract}
\tableofcontents
\section{QGSW equations and main result}
\noindent In this work, we are concerned with the quasi-geostrophic shallow-water equations with a  parameter $\lambda\geqslant 0$, which is a two dimensional active scalar equation   taking the form 
\begin{equation}\label{QGSW equations}
	(\textnormal{QGSW})_{\lambda}\left\lbrace\begin{array}{ll}
		\partial_{t}\boldsymbol{q}+\mathbf{v}\cdot\nabla\boldsymbol{q}=0, &\quad  (t,x)\in\mathbb{R}_{+}\times\mathbb{R}^{2}\\
		\mathbf{v}=\nabla^{\perp}(\Delta-\lambda^{2})^{-1}\boldsymbol{q}, & \\
		\boldsymbol{q}(0,\cdot)=\boldsymbol{q}_{0},&
	\end{array}\right.\quad \mbox{ where }\quad \nabla^{\perp}=\left(\begin{array}{c}
		-\partial_{2}\\
		\partial_{1}
	\end{array}\right).
\end{equation}
 The involved quantities are  the divergence-free velocity field $\mathbf{v}$ and  the potential vorticity $\boldsymbol{q}$ which is a scalar function. The parameter $\lambda$ stands for the inverse Rossby radius defined in the  literature by
$$\lambda=\frac{\omega_{c}}{\sqrt{gH}},$$
where $g$ is the gravity constant, $H$ is the mean active layer depth and $\omega_{c}$ is the Coriolis frequency, assumed to be  constant. Notice that the case $\lambda=0$ corresponds to the velocity-vorticity formulation of Euler equations.
The system \eqref{QGSW equations} is commonly used to track the dynamics  of the atmospheric and oceanic circulation at large scale motion. For  a general review about the asymptotic derivation of the these equations from the rotating shallow water equations  we refer for instance to \cite[p. 220]{V17}.

The main purpose of this paper is to explore the emergence of  time periodic solutions  in the patch form  close to the annulus of radii $1$ and $b$ for the system $(\textnormal{QGSW})_{\lambda}$ with fixed $\lambda>0$ and $b\in(0,1).$  
Recall that  a vortex patch means  a  solution of \eqref{QGSW equations} with initial condition being the characteristic
function of a  bounded domain $D_0\subset\mathbb{R}^2$, that is $\boldsymbol{q}_{0}=\mathbf{1}_{D_{0}}.$
 Actually, this structure is conserved in time due to  the transport structure  of \eqref{QGSW equations}, and one gets 
 $$\boldsymbol{q}(t,\cdot)=\mathbf{1}_{D_{t}}\quad\mbox{ where }\quad D_{t}:=\mathbf{\Phi}_{t}(D_{0}),$$
with $\mathbf{\Phi}_{t}:\mathbb{R}^{2}\rightarrow\mathbb{R}^{2}$ is the flow map associated to $\mathbf{v}$, defined through the ODE 
\begin{equation}\label{definition of the flow map}
\partial_{t}\mathbf{\Phi}_{t}(z)=\mathbf{v}(t,\mathbf{\Phi}_{t}(z))\quad\mbox{ and }\quad\mathbf{\Phi}_{0}=\textnormal{Id}_{\mathbb{R}^{2}}.
\end{equation}
In this framework, of bounded  datum with compact support, existence and uniqueness follow in a standard way from  Yudovich approach implemented  for Euler equations and which can be adapted here in a similar way.   
 These structures can be considered as a toy model to simulate  hurricanes motion in the context of geophysical flows. In the smooth boundary case, their dynamics is completely  described by the evolution of the interfaces surrounding the patch according to the contour dynamics equation given by  
\begin{equation}\label{Lagran-Fo}
	\left[\partial_{t}\gamma(t,\theta)-\mathbf{v}(t,\gamma(t,\theta))\right]\cdot\mathbf{n}(t,\gamma(t,\theta))=0,
\end{equation}
where $\gamma(t,\cdot):\mathbb{T}\rightarrow\partial D_{t}$ is a $C^1$ parametrization of the boundary of the patch and $\mathbf{n}(t,\cdot)$ is an  outward normal vector to the boundary. We may refer to \cite{HMV13,HMV15}  for a detailed derivation of  this equation for active scalar models. We are particularly interested in the existence of ordered structures moving without shape deformation, called \textit{V-states}. More precisely, we shall focus on the existence of  uniformly rotating vortex patches about their center of mass, that can be fixed at the origin,  and with a constant angular velocity $\Omega\in\mathbb{R}$, namely
\begin{equation}\label{v-states}
\boldsymbol{q}(t,\cdot)=\mathbf{1}_{D_t}\quad \mbox{ with }\quad D_{t}=e^{\ii t\Omega}D_{0}.
\end{equation}
In the present work  we explore the case of doubly-connected V-states with $\mathbf{m}$-fold symmetry. To fix the terminology, a bounded open domain $D_0$ is said doubly-connected if 
$$D_{0}=D_{1}\backslash \overline{D_{2}},$$
where $D_{1}$ and ${D_{2}}$ are two bounded open  simply-connected domains with $\overline{D_{2}}\subset D_{1}. $ This means that the boundary of $D_{0}$ is given by two interfaces, one of them is contained in the  open region delimited by the second one. 
According to the structure of $(QGSW)_{\lambda}$,  every radial initial domain $D_0$ generates a trivial stationary solution, and therefore a  V-state rotating with any angular velocity. Basic examples are given by the discs in the simply-connected case or the annuli in the doubly-connected case. 
The first non-trivial examples of uniformly rotating solutions for Euler equations are Kirchhoff ellipses which rotate with the angular velocity $\Omega=\frac{ab}{(a+b)^{2}}$ where $a$ and $b$ are the semi-axes of the ellipse (see \cite{K74} and \cite[p. 304]{BM02}). In \cite{DZ78} Deem and Zabusky established numerically  the existence of simply-connected rotating patches  with $\mathbf{m}$-fold symmetry for $\mathbf{m}>2.$ An analytical proof based on  bifurcation theory and complex analysis tools was performed by Burbea  in \cite{B82} showing the existence of $\mathbf{m}$-fold (for any $\mathbf{m}\in\mathbb{N}^{*}$) symmetric V-states bifurcating from Rankine vortices with angular velocity $\Omega_{\mathbf{m}}:=\frac{\mathbf{m}-1}{2\mathbf{m}}$.
In the spirit of Burbea's work, a lot of results on $\mathbf{m}$-fold V-states have been obtained both for simply and doubly-connected cases for Euler, $(SQG)_{\alpha}$ and $(QGSW)_{\lambda}$ equations in the past decade. We may refer to \cite{CCG16,DHR19,GPSY20,HH15,HHH18,HHHM15,HHMV16,R17}. From this long list, we shall make some comments  on two contributions from \cite{DHR19,HHMV16} related to the current work. In \cite[Thm. B]{HHMV16}, the authors proved for Euler equations that under the condition
$$1+b^{\mathbf{m}}-\frac{\mathbf{m}(1-b^{2})}{2}<0,$$
one can find two branches of  $\mathbf{m}$-fold doubly-connected V-states bifurcating from the normalized  annulus $A_{b}$, defined by
\begin{equation}\label{definition of the annulus Ab}
	A_{b}:=\big\{ z\in\mathbb{C}\quad\textnormal{s.t.}\quad b<|z|<1\big\}\quad\mbox{ for }\quad b\in(0,1).
\end{equation}
at  the following angular velocities 
\begin{equation}\label{def eig Euler}
	\Omega_{\mathbf{m}}^{\pm}(b)=\frac{1-b^{2}}{4}\pm\frac{1}{2\mathbf{m}}\sqrt{\left(\frac{\mathbf{m}(1-b^{2})}{2}-1\right)^{2}-b^{2\mathbf{m}}}.
\end{equation}
Burbea's result has been extended for  $(QGSW)_{\lambda}$ in \cite[Thm. 5.1]{DHR19}, where it is shown the  existence of branches of $\mathbf{m}$-fold symmetric V-states ($\mathbf{m}\geq 2$) bifurcating from Rankine vortex ${\bf{1}}_{\mathbb{D}}$, with $\mathbb{D}$ being the unit disc, at the angular velocitity
\begin{equation}\label{def eig sc QGSW}
	\Omega_{\mathbf{m}}(\lambda)=I_{1}(\lambda)K_{1}(\lambda)-I_{\mathbf{m}}(\lambda)K_{\mathbf{m}}(\lambda)
\end{equation}
where $I_m$ and $K_m$ are the modified Bessel functions of first and second kind, respectively. We may refer to the \mbox{Appendix \ref{appendix Bessel}} for the  definitions and some basic  properties of these functions. We also notice that more analytical and numerical experiments were carefully explored in \cite{DHR19,DJ20} dealing in particular with the imperfect bifurcation and the response of the bifurcation diagram with respect to the parameter $\lambda.$

We emphasize that different studies around this subject have been recently investigated  by several authors, we refer for instance to   \cite{G20,G21,G19,HH21,HW21,HM17} and the references therein.\\ 
The main contribution of this paper is to establish for $(QGSW)_{\lambda}$ the existence of branches of bifurcation in the doubly-connected case, generalizing the result of \cite{HHMV16}. More precisely, we prove the following result.

\begin{theo}\label{main theorem}
Let $\lambda>0$ and $b\in(0,1)$. There exists $N(\lambda,b)\in\mathbb{N}^{*}$ such that for every $\mathbf{m}\in\mathbb{N}^{*},$ with $\mathbf{m}\geqslant N(\lambda,b)$,  there exist two curves of $\mathbf{m}$-fold doubly-connected V-states  bifurcating from the annulus $A_{b}$ defined in \eqref{definition of the annulus Ab},
at the angular velocities
\begin{align*}\Omega_{\mathbf{m}}^{\pm}(\lambda,b)&=\frac{1-b^{2}}{2b}\Lambda_{1}(\lambda,b)+\frac{1}{2}\Big(\Omega_{\mathbf{m}}(\lambda)-\Omega_{\mathbf{m}}(\lambda b)\Big)\\
&\quad\pm\frac{1}{2b}\sqrt{\Big(b\big[\Omega_{\mathbf{m}}(\lambda)+\Omega_{\mathbf{m}}(\lambda b)\big]-(1+b^{2})\Lambda_{1}(\lambda,b)\Big)^{2}-4b^{2}\Lambda_{\mathbf{m}}^{2}(\lambda,b)},
\end{align*}
where $\Omega_{\mathbf{m}}(\lambda)$ is defined in \eqref{def eig sc QGSW} and    
$$\Lambda_{m}(\lambda,b):=I_{m}(\lambda b)K_{m}(\lambda)$$
with $I_m$ and $K_m$ being  the modified Bessel functions of first and second kind. In addition, the boundary of each V-state is of class $C^{1+\alpha},$ for any  $\alpha\in(0,1).$
\end{theo}
Before sketching the proof some remarks are in order.
\begin{remark}
	\begin{enumerate}[label=(\roman*)]
		\item The spectrum is continuous with respect  to $\lambda$ and  $b$. In particular,  when we shrink $\lambda\to 0$ we find the spectrum of Euler equations detailed in \eqref{def eig Euler}. However, when we shrink $b\to 0$ we obtain in part the simply connected spectrum \eqref{def eig sc QGSW} . In other words, 
$$\left\lbrace\begin{array}{l}
	\Omega_{\mathbf{m}}^{\pm}(\lambda,b)\underset{\lambda\rightarrow 0}{\longrightarrow}\Omega_{\mathbf{m}}^{\pm}(b)\\
	\Omega_{\mathbf{m}}^{+}(\lambda,b)\underset{b\rightarrow 0}{\longrightarrow}\Omega_{\mathbf{m}}(\lambda).
\end{array}\right.$$
These asymptotics are obtained for sufficiently large values of $\mathbf{m}.$ For more details see Lemma \ref{lem continuity spectrum}.
		\item  The regularity of the boundary, which  is $C^{1+\alpha}$, is far from being  optimal. We expect it to be analytic and one may use  the approach developed in \cite{CCG16} and successfully implemented in \cite{R17} for the generalized quasi-geostrophic equations in the doubly-connected case.
			\end{enumerate}
\end{remark} 

Now, we intend to  discuss the key steps of the proof of Theorem \ref{main theorem}. The following notation will be used throughout the paper.
\begin{enumerate}[label=\textbullet]
	\item We denote by $\mathbb{D}$ the unit disc. Its boundary, the unit circle, is denoted by $\mathbb{T}$.
	\item Let $f:\mathbb{T}\rightarrow\mathbb{C}$ be a continuous function. We define its mean value by
	$$\fint_{\mathbb{T}}f(\tau)d\tau:=\frac{1}{2\ii\pi}\int_{\mathbb{T}}f(\tau)d\tau:=\frac{1}{2\pi}\int_{0}^{2\pi}f\left(e^{\ii\theta}\right)e^{\ii\theta}d\theta,$$
	where $d\tau$ stands for the complex integration.
\end{enumerate}
First, in Section \ref{sec funct setting}, we reformulate the vortex patch equation by using conformal maps. Indeed, consider an initial doubly-connected domain $D_0=D_{1}\setminus \overline{D_{2}}$, with $D_{1}$ and $D_{2}$ are two simply-connected domains close to the discs of radii $1$ and $b$ respectively. We introduce for $j\in\{1,2\}$ the conformal mappings  $\Phi_j:\mathbb{D}^c\rightarrow D_j^c$ taking the form
$$\Phi_{1}(z)=z+f_1(z)=z+\sum_{n=0}^{\infty}\frac{a_n}{z^n},\quad\Phi_{2}(z)=bz+f_2(z)=bz+\sum_{n=0}^{\infty}\frac{b_n}{z^n}.$$ 
Thus, from the contour dynamics equation, rotating doubly-connected V-states amounts to finding non-trivial zeros of the non-linear functional $G=(G_1,G_2)$, defined for $j\in\{1,2\}$ and $w\in\mathbb{T}$ by
$$G_{j}(\lambda,b,\Omega,f_{1},f_{2})(w):=\mbox{Im}\left\lbrace\Big(\Omega\Phi_{j}(w)+S(\lambda,\Phi_{2},\Phi_{j})(w)-S(\lambda,\Phi_{1},\Phi_{j})(w)\Big)\overline{w}\overline{\Phi_{j}'(w)}\right\rbrace,$$
with
$$\forall w\in\mathbb{T},\quad S(\lambda,\Phi_{i},\Phi_{j})(w):=\fint_{\mathbb{T}}\Phi_{i}'(\tau)K_{0}\left(\lambda|\Phi_{j}(w)-\Phi_{i}(\tau)|\right)d\tau.$$
For this aim, we shall implement   Crandall-Rabinowitz's Theorem, starting from the elementary observation  that the annulus $A_b$ defined by \eqref{definition of the annulus Ab} generates a trivial line of solutions for any $\Omega\in\mathbb{R}$, which will play the role of the bifurcation parameter.  In the same section together with the Appendix \ref{appendix proof reg funct}, we also study the regularity of $G$ and prove that it is of class $C^1$ with respect to the functional spaces introduced in Section \ref{subsec func spaces}. Then, in Section \ref{sec spec std}, we compute the linearized operator at the equilibrium state and prove that it is a Fourier matrix multiplier. More precisely, for 
$$\forall w\in\mathbb{T},\quad h_{1}(w)=\sum_{n=0}^{\infty}a_{n}\overline{w}^{n}\quad \mbox{ and }\quad h_{2}(w)=\sum_{n=0}^{\infty}b_{n}\overline{w}^{n},$$
we have 
$$DG(\lambda,b,\Omega,0,0)[h_1,h_2](w)=\sum_{n=0}^{\infty}(n+1)M_{n+1}(\lambda,b,\Omega)\left(\begin{array}{c}
	a_n\\
	b_n
\end{array}\right)\textnormal{Im}(w^{n+1}),$$
where
$$M_{n}(\lambda,b,\Omega):=\left(\begin{array}{cc}
	\Omega_{n}(\lambda)-\Omega-b\Lambda_{1}(\lambda,b) & b\Lambda_{n}(\lambda,b)\\
	-\Lambda_{n}(\lambda,b) & \Lambda_{1}(\lambda,b)-b\big[\Omega_{n}(\lambda b)+\Omega\big]
\end{array}\right).$$
We refer to Proposition \ref{proposition linearized operator} for more details and point out that some difficulties appear there when computing some integrals related to Bessel functions. Then, the  kernel for the linearized operator $DG(\lambda,b,\Omega,0,0)$ is non trivial  for $\Omega=\Omega_{\mathbf{m}}^{\pm}(\lambda,b)$, as defined in Theorem \ref{main theorem}, with  $\mathbf{m}$ large enough. The restriction to higher  symmetry $\mathbf{m}\geqslant N(\lambda,b)$  is needed first  to ensure  the condition
$$\Delta_{\mathbf{m}}(\lambda,b):=\Big(b\big[\Omega_{\mathbf{m}}(\lambda)+\Omega_{\mathbf{m}}(\lambda b)\big]-(1+b^{2})\Lambda_{1}(\lambda,b)\Big)^{2}-4b^{2}\Lambda_{\mathbf{m}}^{2}(\lambda,b)>0,$$
required  in the transversality condition of Crandall-Rabinowitz's Theorem
and second  to get   the monotonicity of the sequences $\left(\Omega_{n}^{\pm}(\lambda,b)\right)_{n\geqslant N(\lambda,b)}$ (to get a one-dimensional kernel), obtained from tricky  asymptotic analysis  on the  modified Bessel functions. For more details,  we refer to Proposition \ref{proposition assumptions of CR theorem}.
We point out that the  degenerate case  corresponding to $\Delta_{\mathbf{m}}(\lambda,b)=0$ where the transversality is no longer  true was studied in \cite{HM16-2} for Euler equations ($\lambda=0$). It requires to expand the functional at higher order in order to understand the local structure of the bifurcation diagram.  In our case, the dependance of $\Delta_{\mathbf{m}}(\lambda,b)$ with respect to the parameter $b$ is more involved and similar approach may be implemented with a high computational  cost.

\section{Functional settings}\label{sec funct setting}
In this section, we shall reformulate the problem of finding V-states looking at the zeros of a nonlinear functional $G$. We also introduce the function spaces used in the analysis and study some regularity aspects for the functional $G$ with respect to these functions spaces.
\subsection{Boundary equations}
In this subsection we shall obtain the system governing the patch motion. The starting point is the vortex patch equation \eqref{Lagran-Fo}, which writes using the complex notation
\begin{equation}\label{vortex patch equation in the complex sens}
	\mbox{Im}\left\lbrace\big[\partial_{t}\gamma(t,s)-\mathbf{v}(t,\gamma(t,s))\big]\overline{\partial_{s}\gamma(t,s)}\right\rbrace=0,
\end{equation}
where $s\mapsto\gamma(t,s)$ is a parametrization of the boundary of $D_{t}.$ Assuming that the patch is uniformly rotating with an angular velocity $\Omega$, we can choose a parametrization $\gamma$ in the form
\begin{equation}\label{parametrization rotating boundary}
	\gamma(t,s)=e^{\ii\Omega t}\gamma(0,s).
\end{equation}
One readily has 
\begin{equation}\label{simplification vortex patch equation 1}
	\mbox{Im}\left\lbrace\partial_{t}\gamma(t,s)\overline{\partial_{s}\gamma(t,s)}\right\rbrace=\Omega\mbox{Re}\left\lbrace\gamma(0,s)\overline{\partial_{s}\gamma(0,s)}\right\rbrace.
\end{equation}
Now, to study the second term in the equation \eqref{vortex patch equation in the complex sens}, one needs an explicit formulation of the velocity field $\mathbf{v}.$ It has been proved in \cite{DHR19,HR21} that the velocity field associated to $(QGSW)_{\lambda}$ equations writes in the context of vortex patches as an integral on the boundary, namely
\begin{equation}\label{velocity}
	\mathbf{v}(t,z)=\frac{1}{2\pi}\int_{\partial D_{t}}K_{0}(\lambda|z-\xi|)d\xi,
\end{equation}
where the domain $D_t$ is oriented with the convention "matter on the left" due to Stokes' Theorem  and where $K_0$ is the modified Bessel function of second kind. We shall refer to Appendix \ref{appendix Bessel} for the definitions and properties of modified Bessel functions.
By using \eqref{parametrization rotating boundary}, we obtain 
\begin{align*}
	\mathbf{v}(t,\gamma(t,s)) & =  \displaystyle\frac{1}{2\pi}\int_{\partial D_{t}}K_{0}\left(\lambda|\gamma(t,s)-\xi|\right)d\xi\\
	& =  \displaystyle\frac{1}{2\pi}\int_{0}^{1}K_{0}\left(\lambda|e^{\ii\Omega t}\gamma(0,s)-e^{\ii\Omega t}\gamma(0,s')|\right)\partial_{s'}\gamma(t,s')ds'\\
	& =  \displaystyle\frac{e^{\ii\Omega t}}{2\pi}\int_{0}^{1}K_{0}\left(\lambda|\gamma(0,s)-\gamma(0,s')|\right)\partial_{s'}\gamma(0,s')ds'\\
	& = \displaystyle\frac{e^{\ii\Omega t}}{2\pi}\int_{\partial D_{0}}K_{0}\left(\lambda|\gamma(0,s)-\xi|\right)d\xi\\
	& =  e^{\ii\Omega t}\mathbf{v}(0,\gamma(0,s)).
\end{align*}
Consequently using again \eqref{parametrization rotating boundary}, we get
\begin{equation}\label{simplification vortex patch equation 2}
	\mbox{Im}\left\lbrace \mathbf{v}(t,\gamma(t,s))\overline{\partial_{s}\gamma(t,s)}\right\rbrace=\mbox{Im}\left\lbrace \mathbf{v}(0,\gamma(0,s))\overline{\partial_{s}\gamma(0,s)}\right\rbrace.
\end{equation}
Putting together \eqref{simplification vortex patch equation 1} and \eqref{simplification vortex patch equation 2}, the equation \eqref{vortex patch equation in the complex sens} can be rewritten
\begin{equation}\label{fin vpe}
	\Omega\mbox{Re}\left\lbrace\gamma(0,s)\overline{\partial_{s}\gamma(0,s)}\right\rbrace=\mbox{Im}\left\lbrace \mathbf{v}(0,\gamma(0,s))\overline{\partial_{s}\gamma(0,s)}\right\rbrace.
\end{equation}
Let us assume that our starting domain $D_{0}$ is doubly-connected, that is
$$D_{0}=D_{1}\backslash \overline{D_{2}}\quad\textnormal{with}\quad \overline{D}_{2}\subset D_{1},$$
where $D_{1}$ and $D_{2}$ are simply-connected bounded open domains of $\mathbb{C}.$ Then combining \eqref{velocity} and \eqref{fin vpe}, one should obtain for all $z\in\partial D_{0}=\partial D_{1}\cup\partial D_{2},$
\begin{align}\label{boundary equation for doubly-connected QGSW}
	\Omega\mbox{Re}\left\lbrace z\overline{z'}\right\rbrace & =  \displaystyle\mbox{Im}\left\lbrace\frac{1}{2\pi}\int_{\partial D_{0}}K_{0}\left(\lambda|z-\xi|\right)d\xi\overline{z'}\right\rbrace\nonumber\\
	& =  \displaystyle\mbox{Im}\left\lbrace\left(\frac{1}{2\pi}\int_{\partial D_{1}}K_{0}\left(\lambda|z-\xi|\right)d\xi-\frac{1}{2\pi}\int_{\partial D_{2}}K_{0}\left(\lambda|z-\xi|\right)d\xi\right)\overline{z'}\right\rbrace,
\end{align}
where $z'$ denotes a tangent vector to the boundary $\partial D_{0}$ at the point $z.$ The minus sign in front of the integral on $\partial D_{2}$ is here because of the orientation convention for the application of Stokes' Theorem. Following the works initiated by Burbea, see for instance \cite{B82,DHR19,HMV13,HMV15}, we should give the equation(s) to solve by using conformal mappings. For this purpose, we shall recall Riemann mapping Theorem. 
\begin{theo}[Riemann Mapping]
	Let $\mathbb{D}$ denote the unit open ball and $D_{0}\subset\mathbb{C}$ be a simply
		connected bounded domain. Then there exists a unique bi-holomorphic map called also conformal
		map, $\Phi:\mathbb{C}\backslash\overline{\mathbb{D}}\rightarrow\mathbb{C}\backslash\overline{D_{0}}$ taking the form
		$$\Phi(z)=az+\sum_{n=0}^{\infty}\frac{a_{n}}{z^{n}},$$
		with $a>0$ and $(a_{n})_{n\in\mathbb{N}}\in\mathbb{C}^{\mathbb{N}}.$
\end{theo}
Notice that in the previous theorem, the domain is only assumed to be simply-connected and bounded. In particular, the existence of the conformal mapping does not depend on the regularity of the boundary. However, information on the regularity of the conformal mapping implies some regularity of the boundary. This is given by the following result which can be found in \cite{W35} or in \cite[Thm. 3.6]{P92}.
\begin{theo}[Kellogg-Warschawski]
	We keep the notations of Riemann mapping Theorem. If the conformal map $\Phi:\mathbb{C}\backslash\overline{\mathbb{D}}\rightarrow\mathbb{C}\backslash\overline{D_{0}}$ has a continuous extension to $\mathbb{C}\backslash\mathbb{D}$ which is of class $C^{n+1+\beta}$ with $n\in\mathbb{N}$ and $\beta\in(0,1)$, then the boundary $\Phi(\mathbb{T})$ is a Jordan curve of class $C^{n+1+\beta}.$
\end{theo}
Assuming that $D_{1}$ and $D_{2}$ are respectively small deformations of the discs of radii $1$ and $b$, so that the shape of $D_{0}$ is close to the annulus $A_b$ defined in \eqref{definition of the annulus Ab}, we shall consider the parametrizations by the conformal mapping 
$\Phi_{j}:\mathbb{C}\backslash\overline{\mathbb{D}}\rightarrow\mathbb{C}\backslash\overline{D_{j}}$ satisfying
$$\Phi_{1}(z)=z+f_{1}(z)=z\left(1+\sum_{n=1}^{\infty}\frac{a_{n}}{z^{n}}\right)$$
and
$$\Phi_{2}(z)=bz+f_{2}(z)=z\left(b+\sum_{n=1}^{\infty}\frac{b_{n}}{z^{n}}\right).$$ 
We shall now rewrite the equations by using the conformal parametrizations $\Phi_{1}$ and $\Phi_{2}.$ First remark that for $w\in\mathbb{T},$ a tangent vector on the boundary $\partial D_{j}$ at the point $z=\Phi_{j}(w)$ is given by
$$\overline{z'}=-\ii\overline{w}\overline{\Phi_{j}'(w)}.$$
Inserting this into \eqref{boundary equation for doubly-connected QGSW} and using the change of variables $\xi=\Phi_{j}(\tau)$ gives
$$\forall j\in\{1,2\},\quad\forall w\in\mathbb{T},\quad G_{j}(\lambda,b,\Omega,f_{1},f_{2})(w)=0,$$
where
\begin{equation}\label{definition of Fj}
	G_{j}(\lambda,b,\Omega,f_{1},f_{2})(w):=\mbox{Im}\left\lbrace\Big(\Omega\Phi_{j}(w)+S(\lambda,\Phi_{2},\Phi_{j})(w)-S(\lambda,\Phi_{1},\Phi_{j})(w)\Big)\overline{w}\overline{\Phi_{j}'(w)}\right\rbrace,
\end{equation}
with
\begin{equation}\label{def S}
	\forall(i,j)\in\{1,2\}^{2},\quad\forall w\in\mathbb{T},\quad S(\lambda,\Phi_{i},\Phi_{j})(w):=\fint_{\mathbb{T}}\Phi_{i}'(\tau)K_{0}\left(\lambda|\Phi_{j}(w)-\Phi_{i}(\tau)|\right)d\tau.
\end{equation}
Then, finding a non trivial uniformly rotating vortex patch for \eqref{QGSW equations} reduces to finding zeros of the non-linear functional 
$$G:=(G_{1},G_{2}).$$
As stated in the introduction, these non trivial solutions may be obtained by bifurcation techniques from trivial solutions which are annuli. Let us recover with this formalism that indeed the annuli rotate for any angular velocity. This is given by the following result.
\begin{lem}\label{lemma annulus as trivial solution}
	Let $b\in(0,1).$ Then the annulus $A_{b}$ defined in \eqref{definition of the annulus Ab} is a rotating patch for \eqref{QGSW equations} for any angular velocity $\Omega\in\mathbb{R}.$
\end{lem}
\begin{proof}
	Taking $f_{1}=f_{2}=0$ by in \eqref{definition of Fj}, we get 
	$$G_{1}(\lambda,b,\Omega,0,0)(w)=\mbox{Im}\left\lbrace b\overline{w}\fint_{\mathbb{T}}K_{0}\left(\lambda|w-b\tau|\right)d\tau-\overline{w}\fint_{\mathbb{T}}K_{0}\left(\lambda|w-\tau|\right)d\tau\right\rbrace.$$
	Using the changes of variables $\tau\mapsto w\tau$ and the fact that $|w|=1$, we have 
	$$G_{1}(\lambda,b,\Omega,0,0)(w)=\mbox{Im}\left\lbrace b\fint_{\mathbb{T}}K_{0}\left(\lambda|1-b\tau|\right)d\tau-\fint_{\mathbb{T}}K_{0}\left(\lambda|1-\tau|\right)d\tau\right\rbrace=0.$$
	Indeed for $a\in\{1,b\},$ we have by \eqref{symmetry Bessel} and the change of variables $\theta\mapsto-\theta$
	\begin{align}\label{realint}
		\displaystyle\overline{\fint_{\mathbb{T}}K_{0}\left(\lambda|1-a\tau|\right)d\tau} & =  \displaystyle\overline{\frac{1}{2\pi}\int_{0}^{2\pi}K_{0}\left(\lambda|1-ae^{\ii\theta}|\right)e^{\ii\theta}d\theta}\nonumber\\
		& =  \displaystyle\frac{1}{2\pi}\int_{0}^{2\pi}K_{0}\left(\lambda|1-ae^{\ii\theta}|\right)e^{-\ii\theta}d\theta\nonumber\\
		&= \displaystyle\frac{1}{2\pi}\int_{0}^{2\pi}K_{0}\left(\lambda|1-ae^{-\ii\theta}|\right)e^{\ii\theta}d\theta\nonumber\\
		& =  \displaystyle\frac{1}{2\pi}\int_{0}^{2\pi}K_{0}\left(\lambda|1-ae^{\ii\theta}|\right)e^{\ii\theta}d\theta\nonumber\\
		& =  \displaystyle\fint_{\mathbb{T}}K_{0}\left(\lambda|1-a\tau|\right)d\tau.
	\end{align}
	Similarly, we find
	$$G_{2}(\lambda,b,\Omega,0,0)(w)=0.$$
	This proves Lemma \ref{lemma annulus as trivial solution}.
\end{proof}
\subsection{Function spaces and regularity of the functional}\label{subsec func spaces}
We introduce here the function spaces used along this work. Throughout the paper it is more convenient to think of $2\pi$-periodic function $g:\mathbb{R}\rightarrow\mathbb{C}$ as a function of the complex variable $w=e^{\ii\theta}.$ To be more precise, let $f:\mathbb{T}\rightarrow\mathbb{R}^{2}$, be a continuous function, then it can be assimilated to a $2\pi$-periodic function $g:\mathbb{R}\rightarrow\mathbb{R}^{2}$ via the relation
$$f(w)=g(\theta),\quad w=e^{\ii\theta}.$$
Hence, when $f$ is smooth enough, we get
$$f'(w):=\frac{df}{dw}=-\ii e^{-\ii\theta}g'(\theta).$$
Since $\frac{d}{dw}$ and $\frac{d}{d\theta}$ differ only by a smooth factor with modulus one, we shall in the
sequel work with $\frac{d}{dw}$ instead of $\frac{d}{d\theta}$ which appears more suitable in the computations.
In addition, if $f$ is of class $C^{1}$ and has real Fourier coefficients, then we can easily check that
\begin{equation}\label{formula derivative conjugate}
\left(\overline{f}\right)'(w)=-\frac{\overline{f'(w)}}{w^{2}}.
\end{equation}
We shall now recall the definition of Hölder spaces on the unit circle.
\begin{defin}
Let $\alpha\in(0,1).$
\begin{enumerate}[label=(\roman*)]
\item We denote by $C^{\alpha}(\mathbb{T})$ the space of continuous functions $f$ such that
$$\| f\|_{C^{\alpha}(\mathbb{T})}:=\| f\|_{L^{\infty}(\mathbb{T})}+\sup_{\underset{\tau\neq w}{(\tau,w)\in\mathbb{T}^{2}}}\frac{|f(\tau)-f(w)|}{|\tau-w|^{\alpha}}<\infty.$$
\item We denote by $C^{1+\alpha}(\mathbb{T})$ the space of $C^{1}$ functions with $\alpha$-Hölder continuous derivative 
$$\| f\|_{C^{1+\alpha}(\mathbb{T})}:=\| f\|_{L^{\infty}(\mathbb{T})}+\Big\| \frac{df}{dw}\Big\|_{C^{\alpha}(\mathbb{T})}<\infty.$$
\end{enumerate}
\end{defin}
\noindent For $\alpha\in(0,1),$ we set
$$X^{1+\alpha}:=X_{1}^{1+\alpha}\times X_{1}^{1+\alpha}\quad\mbox{ with }\quad X_{1}^{1+\alpha}:=\left\lbrace f\in C^{1+\alpha}(\mathbb{T})\quad\textnormal{s.t.}\quad\forall w\in\mathbb{T},\,f(w)=\sum_{n=0}^{\infty}f_{n}\overline{w}^{n},\,f_{n}\in\mathbb{R}\right\rbrace$$
and
$$Y^{\alpha}:=Y_{1}^{\alpha}\times Y_{1}^{\alpha}\quad\mbox{ with }\quad Y_{1}^{\alpha}:=\left\lbrace g\in C^{\alpha}(\mathbb{T})\quad\textnormal{s.t.}\quad\forall w\in\mathbb{T},\,g(w)=\sum_{n=1}^{\infty}g_{n}e_{n}(w),\,g_{n}\in\mathbb{R}\right\rbrace,$$
where
$$e_{n}(w):=\mbox{Im}(w^{n}).$$
We denote 
$$B_{r}^{1+\alpha}:=\Big\{ f\in X_{1}^{1+\alpha}\quad\textnormal{s.t.}\quad\| f\|_{C^{1+\alpha}(\mathbb{T})}<r\Big\}.$$
We can encode the $\mathbf{m}$-fold structure in the functional spaces by setting
$$X_{\mathbf{m}}^{1+\alpha}:=X_{1,\mathbf{m}}^{1+\alpha}\times X_{1,\mathbf{m}}^{1+\alpha}\quad \mbox{ with }\quad X_{1,\mathbf{m}}^{1+\alpha}:=\left\lbrace f\in X_{1}^{1+\alpha}\quad\textnormal{s.t.}\quad\forall w\in\mathbb{T},\,f(w)=\sum_{n=1}^{\infty}f_{\mathbf{m}n-1}\overline{w}^{\mathbf{m}n-1}\right\rbrace$$
and
$$Y_{\mathbf{m}}^{\alpha}:=Y_{1,\mathbf{m}}^{\alpha}\times Y_{1,\mathbf{m}}^{\alpha}\quad \mbox{ with }\quad Y_{1,\mathbf{m}}^{\alpha}:=\left\lbrace g\in Y_{1}^{\alpha}\quad\textnormal{s.t.}\quad\forall w\in\mathbb{T},\,g(w)=\sum_{n=1}^{\infty}g_{\mathbf{m}n}e_{\mathbf{m}n}(w)\right\rbrace.$$
The spaces $X^{1+\alpha}$ and $X_{\mathbf{m}}^{1+\alpha}$ $\big($resp. $Y^{\alpha}$ and $Y_{\mathbf{m}}^{\alpha}\big)$ are equipped with the strong product topology of $C^{1+\alpha}(\mathbb{T})\times C^{1+\alpha}(\mathbb{T})$ $\big($resp. $C^{\alpha}(\mathbb{T})\times C^{\alpha}(\mathbb{T})\big)$. We also denote 
$$B_{r,\mathbf{m}}^{1+\alpha}:=\Big\{f\in X_{1,\mathbf{m}}^{1+\alpha}\quad\textnormal{s.t.}\quad\| f\|_{C^{1+\alpha}(\mathbb{T})}<r\Big\}=B_{r}^{1+\alpha}\cap X_{1,\mathbf{m}}^{1+\alpha}.$$

We shall now investigate the regularity of the nonlinear functional $G$ defined by \eqref{definition of Fj}. Indeed, Crandall-Rabinowitz's Theorem \ref{Crandall-Rabinowitz theorem} requires some regularity assumptions to apply and this is what we check here. The ingredients of the proof are classical and they are  postponed to the Appendix \ref{appendix proof reg funct}.
\begin{prop}\label{proposition regularity of the functional}
Let $\lambda>0,$ $b\in(0,1)$, $\alpha\in(0,1)$ and $\mathbf{m}\in\mathbb{N}^*.$ There exists $r>0$ such that 
\begin{enumerate}[label=(\roman*)]
\item $G(\lambda,b,\cdot,\cdot,\cdot):\mathbb{R}\times B_{r}^{1+\alpha}\times B_{r}^{1+\alpha}\rightarrow Y^{\alpha}$ is well-defined and of classe $C^{1}.$
\item The restriction $G(\lambda,b,\cdot,\cdot,\cdot):\mathbb{R}\times B_{r,\mathbf{m}}^{1+\alpha}\times B_{r,\mathbf{m}}^{1+\alpha}\rightarrow Y_{\mathbf{m}}^{\alpha}$ is well-defined.
\item The partial derivative $\partial_{\Omega}DG(\lambda,b,\cdot,\cdot,\cdot):\mathbb{R}\times B_{r}^{1+\alpha}\times B_{r}^{1+\alpha}\rightarrow\mathcal{L}(X^{1+\alpha},Y^{\alpha})$ exists and is continuous.
\end{enumerate}
\end{prop}
\section{Spectral study}\label{sec spec std}
In this section, we study the linearized operator at the equilibrium state and look for the degeneracy conditions for its kernel.
\subsection{Linearized operator}
In this subsection, we compute the differential $DG(\lambda,b,\Omega,0,0)$ and show that it acts as a Fourier multiplier. More precisely, we prove the following proposition.
\begin{prop}\label{proposition linearized operator}
Let $\lambda>0,$ $b\in(0,1)$ and $\alpha\in(0,1).$ Then for all $\Omega\in\mathbb{R}$ and for all $(h_{1},h_{2})\in X^{1+\alpha},$ if we write
$$h_{1}(w)=\sum_{n=0}^{\infty}a_{n}\overline{w}^{n}\quad \mbox{ and }\quad h_{2}(w)=\sum_{n=0}^{\infty}b_{n}\overline{w}^{n},$$
we have for all $w\in\mathbb{T}$ 
$$DG(\lambda,b,\Omega,0,0)(h_{1},h_{2})(w)=\sum_{n=0}^{\infty}(n+1)M_{n+1}(\lambda,b,\Omega)\left(\begin{array}{c}
a_{n}\\
b_{n}
\end{array}\right)e_{n+1}(w),$$
where for all $n\in\mathbb{N}^{*}$, the matrix $M_{n}(\lambda,b,\Omega)$ is defined by
$$M_{n}(\lambda,b,\Omega):=\left(\begin{array}{cc}
\Omega_{n}(\lambda)-\Omega-b\Lambda_{1}(\lambda,b) & b\Lambda_{n}(\lambda,b)\\
-\Lambda_{n}(\lambda,b) & \Lambda_{1}(\lambda,b)-b\big[\Omega_{n}(\lambda b)+\Omega\big]
\end{array}\right),$$
with
$$\Lambda_{n}(\lambda,b):=I_{n}(\lambda b)K_{n}(\lambda)$$
and
$$\forall x>0,\quad\Omega_{n}(x):=I_{1}(x)K_{1}(x)-I_{n}(x)K_{n}(x).$$
Recall that the modified Bessel functions $I_n$ and $K_n$ are defined in Appendix \ref{appendix Bessel}.
\end{prop}
\begin{proof}
Since $G=(G_{1},G_{2}),$ then for given $(h_{1},h_{2})\in X^{1+\alpha}$, we have 
\begin{equation}\label{DG at 0}
DG(\lambda,b,\Omega,0,0)(h_{1},h_{2})=\left(\begin{array}{c}
D_{f_{1}}G_{1}(\lambda,b,\Omega,0,0)h_{1}+D_{f_{2}}G_{1}(\lambda,b,\Omega,0,0)h_{2}\\
D_{f_{1}}G_{2}(\lambda,b,\Omega,0,0)h_{1}+D_{f_{2}}G_{2}(\lambda,b,\Omega,0,0)h_{2}
\end{array}\right).
\end{equation}
But, with the notation introduced in Appendix \ref{appendix proof reg funct}, we can write
\begin{equation}\label{DGis at 0}
\left\lbrace\begin{array}{rcl}
D_{f_{1}}G_{1}(\lambda,b,\Omega,0,0)h_{1} & = & D_{f_{1}}\mathcal{S}_{1}(\lambda,b,\Omega,0)h_{1}+D_{f_{1}}\mathcal{I}_{1}(\lambda,b,0,0)h_{1}\\
D_{f_{2}}G_{2}(\lambda,b,\Omega,0,0)h_{2} & = & D_{f_{2}}\mathcal{S}_{2}(\lambda,b,\Omega,0)h_{2}+D_{f_{2}}\mathcal{I}_{2}(\lambda,b,0,0)h_{2}\\
D_{f_{2}}G_{1}(\lambda,b,\Omega,0,0)h_{2} & = & D_{f_{2}}\mathcal{I}_{1}(\lambda,b,0,0)h_{2}\\
D_{f_{1}}G_{2}(\lambda,b,\Omega,0,0)h_{1} & = & D_{f_{1}}\mathcal{I}_{2}(\lambda,b,0,0)h_{1}.
\end{array}\right.
\end{equation}
We write
$$h_{1}(w)=\sum_{n=0}^{\infty}a_{n}\overline{w}^{n}\quad \mbox{ and }\quad h_{2}(w)=\sum_{n=0}^{\infty}b_{n}\overline{w}^{n}.$$
It has already been proved in \cite[Prop. 5.8]{DHR19} that for all $w\in\mathbb{T},$
\begin{equation}\label{Df1O1 at 0}
D_{f_{1}}\mathcal{S}_{1}(\lambda,b,\Omega,0)h_{1}(w)=\sum_{n=0}^{\infty}(n+1)\left(\Omega_{n+1}(\lambda)-\Omega\right)a_{n}e_{n+1}(w),
\end{equation}
where
$$\Omega_{n}(\lambda):=I_{1}(\lambda)K_{1}(\lambda)-I_{n}(\lambda)K_{n}(\lambda).$$
By a similar calculus, we get 
\begin{equation}\label{Df2O2 at 0}
D_{f_{2}}\mathcal{S}_{2}(\lambda,b,\Omega,0)h_{2}(w)=-\sum_{n=0}^{\infty}(n+1)b\left(\Omega_{n+1}(\lambda b)+\Omega\right)b_{n}e_{n+1}(w).
\end{equation}
In view of \eqref{DfjNj general}, we can write 
$$D_{f_{1}}\mathcal{I}_{1}(\lambda,b,0,0)h_{1}(w)=\mathcal{L}_{1}(h_{1})(w)+\mathcal{L}_{2}(h_{1})(w),$$
with
\begin{align*}
\mathcal{L}_{1}(h_{1})(w)&:=\mbox{Im}\left\lbrace\overline{w}\overline{h_{1}'(w)}b\fint_{\mathbb{T}}K_{0}\left(\lambda|w-b\tau|\right)d\tau\right\rbrace,\\
\mathcal{L}_{2}(h_{1})(w)&:=\mbox{Im}\left\lbrace\frac{\lambda b}{2}\overline{w}\fint_{\mathbb{T}}K_{0}'\left(\lambda|w-b\tau|\right)\frac{\overline{h_{1}(w)}(w-b\tau)+h_{1}(w)(\overline{w}-b\overline{\tau})}{|w-b\tau|}d\tau\right\rbrace.
\end{align*}
By using the change of variables $\tau\mapsto w\tau$ and the fact that $|w|=1$, we deduce 
$$\overline{w}\fint_{\mathbb{T}}K_{0}\left(\lambda|w-b\tau|\right)d\tau=\fint_{\mathbb{T}}K_{0}\left(\lambda|1-b\tau|\right)d\tau.$$
Moreover, from \eqref{realint}, we know that 
\begin{align*}
\fint_{\mathbb{T}}K_{0}\left(\lambda|1-b\tau|\right)d\tau\in\mathbb{R}.
\end{align*}
So using that
\begin{equation}\label{module formula}
	|1-be^{\ii\theta}|=\left(1-2b\cos(\theta)+b^{2}\right)^{\frac{1}{2}}\quad \mbox{ with }\quad b\in(0,1),
\end{equation}
we obtain from \eqref{symmetry Bessel},
\begin{align*}
\displaystyle\fint_{\mathbb{T}}K_{0}\left(\lambda|1-b\tau|\right)d\tau & =  \displaystyle\mbox{Re}\left\lbrace\frac{1}{2\pi}\int_{0}^{2\pi}K_{0}\left(\lambda|1-be^{\ii\theta}|\right)e^{\ii\theta}d\theta\right\rbrace\\
& =  \displaystyle\frac{1}{2\pi}\int_{0}^{2\pi}K_{0}\left(\lambda|1-be^{\ii\theta}|\right)\cos(\theta)d\theta.
\end{align*}
Now, by \eqref{Beltrami's summation formula} and \eqref{symmetry Bessel}, one obtains for all $n\in\mathbb{N}^{*},$
\begin{align}\label{simpl int}
\frac{1}{2\pi}\int_{0}^{2\pi}K_{0}\left(\lambda|1-be^{\ii\theta}|\right)\cos(n\theta)d\theta&=  \frac{1}{2\pi}\int_{0}^{2\pi}\sum_{m=-\infty}^{\infty}I_{m}(\lambda b)K_{m}(\lambda)\cos(m\theta)\cos(n\theta)d\theta\nonumber\\
&=\frac{1}{2\pi}\sum_{m=-\infty}^{\infty}I_{m}(\lambda b)K_{m}(\lambda)\int_{0}^{2\pi}\cos(m\theta)\cos(n\theta)d\theta\nonumber\\
&= I_{n}(\lambda b)K_{n}(\lambda).
\end{align}
Notice that the inversion of symbols of summation and integration is possible due to the geometric decay at infinity given by \eqref{asymptotic expansion of high order for the product InKn}. 
Then, we deduce by \eqref{formula derivative conjugate} that
$$\mathcal{L}_{1}(h_{1})(w)=-\sum_{n=0}^{\infty}nbI_{1}(\lambda b)K_{1}(\lambda)a_{n}e_{n+1}(w).$$
By using the change of variables $\tau\mapsto w\tau$ and the fact that $|w|=1$, we infer 
\begin{align*}
	&\overline{w}\fint_{\mathbb{T}}K_{0}'\left(\lambda|w-b\tau|\right)\frac{\overline{h_{1}(w)}(w-b\tau)+h_{1}(w)(\overline{w}-b\overline{\tau})}{|w-b\tau|}d\tau\\
	&\quad=\fint_{\mathbb{T}}K_{0}'\left(\lambda|1-b\tau|\right)\frac{\overline{h_{1}(w)}w(1-b\tau)+h_{1}(w)\overline{w}(1-b\overline{\tau})}{|1-b\tau|}d\tau.
\end{align*}
But
$$\fint_{\mathbb{T}}K_{0}'\left(\lambda|1-b\tau|\right)\frac{h_{1}(w)\overline{w}(1-b\overline{\tau})}{|1-b\tau|}d\tau=\sum_{n=0}^{\infty}a_{n}\left(\fint_{\mathbb{T}}K_{0}'\left(\lambda|1-b\tau|\right)\frac{(1-b\overline{\tau})}{|1-b\tau|}d\tau\right)\overline{w}^{n+1}$$
and
$$\fint_{\mathbb{T}}K_{0}'\left(\lambda|1-b\tau|\right)\frac{\overline{h_{1}(w)}w(1-b\tau)}{|1-b\tau|}d\tau=\sum_{n=0}^{\infty}a_{n}\left(\fint_{\mathbb{T}}K_{0}'\left(\lambda|1-b\tau|\right)\frac{(1-b\tau)}{|1-b\tau|}d\tau\right)w^{n+1}.$$
Moreover, by writting the line integral with the parametrization $\tau=e^{i\theta}$ and making the change of variables $\theta\mapsto-\theta$, we get as in \eqref{realint} 
$$\displaystyle\fint_{\mathbb{T}}K_{0}'\left(\lambda|1-b\tau|\right)\frac{(1-b\tau)}{|1-b\tau|}d\tau \in\mathbb{R}\quad\mbox{ and }\quad\fint_{\mathbb{T}}K_{0}'\left(\lambda|1-b\tau|\right)\frac{(1-b\overline{\tau})}{|1-b\tau|}d\tau\in\mathbb{R}.$$
Since $\mbox{Im}\left(\overline{w}^{n+1}\right)=-\mbox{Im}\left(w^{n+1}\right),$ we obtain 
$$\mathcal{L}_{2}(h_{1})(w)=\displaystyle\sum_{n=0}^{\infty}a_{n}\left(\frac{\lambda b}{2}\fint_{\mathbb{T}}K_{0}'\left(\lambda|1-b\tau|\right)\frac{b(\overline{\tau}-\tau)}{|1-b\tau|}d\tau\right)\mbox{Im}(w^{n+1}).$$
An integration by parts together with \eqref{module formula} and \eqref{simpl int} gives 
\begin{align*}
\displaystyle\frac{\lambda b}{2}\fint_{\mathbb{T}}K_{0}'\left(\lambda|1-b\tau|\right)\frac{b(\overline{\tau}-\tau)}{|1-b\tau|}d\tau & =  \displaystyle\frac{\lambda b}{4\pi}\int_{0}^{2\pi}K_{0}'\left(\lambda|1-be^{\ii\theta}|\right)\frac{b(e^{-\ii\theta}-e^{\ii\theta})e^{i\theta}}{|1-be^{\ii\theta}|}d\theta\\
& =  \displaystyle\frac{-b}{2\pi}\int_{0}^{2\pi}K_{0}\left(\lambda|1-be^{\ii\theta}|\right)e^{\ii\theta}d\theta\\
& =  \displaystyle\frac{-b}{2\pi}\int_{0}^{2\pi}K_{0}\left(\lambda|1-be^{\ii\theta}|\right)\cos(\theta)d\theta\\
& = \displaystyle-bI_{1}(\lambda b)K_{1}(\lambda).
\end{align*}
Therefore, 
$$\mathcal{L}_{2}(h_{1})(w)=-\sum_{n=0}^{\infty}bI_{1}(\lambda b)K_{1}(\lambda)a_{n}e_{n+1}(w).$$
Finally,
\begin{equation}\label{Df1N1 at 0}
D_{f_{1}}\mathcal{I}_{1}(\lambda,b,0,0)h_{1}(w)=-\sum_{n=0}^{\infty}b(n+1)I_{1}(\lambda b)K_{1}(\lambda)a_{n}e_{n+1}(w).
\end{equation}
Similar computations taking into acount the modification with $b$, change of signs and the fact that $|b-e^{i\theta}|=|1-be^{i\theta}|$ yield
\begin{equation}\label{Df2N2 at 0}
D_{f_{2}}\mathcal{I}_{2}(\lambda,b,0,0)(h_{2})(w)=\sum_{n=0}^{\infty}(n+1)I_{1}(\lambda b)K_{1}(\lambda)b_{n}e_{n+1}(w).
\end{equation}
According to \eqref{DfiNj general}, we can write 
$$D_{f_{2}}\mathcal{I}_{1}(\lambda,b,0,0)h_{2}(w)=\mathcal{L}_{3}(h_{2})(w)+\mathcal{L}_{4}(h_{2})(w),$$
with
\begin{align*}
	\mathcal{L}_{3}(h_{2})(w)&:=\mbox{Im}\left\lbrace\overline{w}\fint_{\mathbb{T}}h_{2}'(\tau)K_{0}\left(\lambda|w-b\tau|\right)d\tau\right\rbrace,\\
	\mathcal{L}_{4}(h_{2})(w)&:=-\frac{\lambda b}{2}\mbox{Im}\left\lbrace\overline{w}\fint_{\mathbb{T}}K_{0}'\left(\lambda|w-b\tau|\right)\frac{\overline{h_{2}(\tau)}(w-b\tau)+h_{2}(\tau)(\overline{w}-b\overline{\tau})}{|w-b\tau|}d\tau\right\rbrace.
\end{align*}
The change of variables $\tau\mapsto w\tau$ implies 
\begin{align*}
\mathcal{L}_{3}(h_{2})(w) & =  \displaystyle\mbox{Im}\left\lbrace\fint_{\mathbb{T}}h_{2}'(w\tau)K_{0}\left(\lambda|1-b\tau|\right)d\tau\right\rbrace\\
& =  \displaystyle-\sum_{n=0}^{\infty}nb_{n}\left(\fint_{\mathbb{T}}\overline{\tau}^{n+1}K_{0}\left(\lambda|1-b\tau|\right)d\tau\right)\mbox{Im}(\overline{w}^{n+1})\\
& =  \displaystyle\sum_{n=0}^{\infty}nb_{n}\left(\fint_{\mathbb{T}}\overline{\tau}^{n+1}K_{0}\left(\lambda|1-b\tau|\right)d\tau\right) e_{n+1}(w).
\end{align*}
But by symmetry and \eqref{simpl int}
\begin{align*}
\displaystyle\fint_{\mathbb{T}}\overline{\tau}^{n+1}K_{0}\left(\lambda|1-b\tau|\right)d\tau & = \displaystyle\frac{1}{2\pi}\int_{0}^{2\pi}e^{-\ii(n+1)\theta}K_{0}\left(\lambda|1-be^{\ii\theta}|\right)e^{\ii\theta}d\theta\\
& =  \displaystyle\frac{1}{2\pi}\int_{0}^{2\pi}K_{0}\left(\lambda|1-be^{\ii\theta}|\right)\cos(n\theta)d\theta\\
& = I_{n}(\lambda b)K_{n}(\lambda).
\end{align*}
Hence,
$$\mathcal{L}_{3}(h_{2})(w)=\sum_{n=0}^{\infty}nI_{n}(\lambda b)K_{n}(\lambda)b_{n}e_{n+1}(w).$$
By using the change of variables $\tau\mapsto w\tau$ and the fact that $|w|=1$, we have 
$$\mathcal{L}_{4}(h_{2})(w)=\frac{-\lambda b}{2}\mbox{Im}\left\lbrace\fint_{\mathbb{T}}K_{0}'\left(\lambda|1-b\tau|\right)\frac{\overline{h_{2}(w\tau)}w(1-b\tau)+h_{2}(w\tau)\overline{w}(1-b\overline{\tau})}{|1-b\tau|}d\tau\right\rbrace,$$
which also writes
$$\mathcal{L}_{4}(h_{2})(w)=\frac{-\lambda b}{2}\sum_{n=0}^{\infty}b_{n}\left(\fint_{\mathbb{T}}K_{0}'\left(\lambda|1-b\tau|\right)\frac{
(\tau^{n}-\overline{\tau}^{n})-b(\tau^{n+1}-\overline{\tau}^{n+1})}{|1-b\tau|}d\tau\right)\textnormal{Im}(w^{n}).$$
We denote 
$$\mathtt{I}:=\frac{-\lambda b}{2}\fint_{\mathbb{T}}K_{0}'\left(\lambda|1-b\tau|\right)\frac{(\tau^{n}-\overline{\tau}^{n})-b(\tau^{n+1}-\overline{\tau}^{n+1})}{|1-b\tau|}d\tau.$$
Since $\mathtt{I}\in\mathbb{R},$ we have 
\begin{align*}
	\mathtt{I}&=\displaystyle\frac{-\lambda b}{4\pi}\int_{0}^{2\pi}K_{0}'\left(\lambda|1-be^{\ii\theta}|\right)\frac{(e^{\ii n\theta}-e^{-\ii n\theta})-b(e^{\ii(n+1)\theta}-e^{-\ii(n+1)\theta})}{|1-be^{\ii\theta}|}e^{\ii\theta}d\theta\\
&=\displaystyle\frac{\lambda b}{2\pi}\int_{0}^{2\pi}K_{0}'\left(\lambda|1-be^{\ii\theta}|\right)\frac{\sin(\theta)}{|1-be^{\ii\theta}|}(\sin(n\theta)-b\sin((n+1)\theta))d\theta.
\end{align*}
Integrating by parts with \eqref{module formula} and using \eqref{simpl int} yield 
\begin{align*}
\mathtt{I}&=\frac{1}{2\pi}\int_{0}^{2\pi}K_{0}(\lambda|1-be^{\ii\theta}|)\left(b(n+1)\cos((n+1)\theta)-n\cos(n\theta)\right)\\
&=b(n+1)I_{n+1}(\lambda b)K_{n+1}(\lambda)-nI_{n}(\lambda b)K_{n}(\lambda).
\end{align*}
Therefore,
\begin{equation}\label{Df2N1 at 0}
D_{f_{2}}\mathcal{I}_{1}(\lambda,b,0,0)(h_{2})(w)=\sum_{n=0}^{\infty}b(n+1)I_{n+1}(\lambda b)K_{n+1}(\lambda)b_{n}e_{n+1}(w).
\end{equation}
Similar computations taking into acount the modification with $b$, change of signs and the fact that $|b-e^{i\theta}|=|1-be^{i\theta}|$ imply
\begin{equation}\label{Df1N2 at 0}
D_{f_{1}}\mathcal{I}_{2}(\lambda,b,0,0)(h_{1})(w)=-\sum_{n=0}^{\infty}(n+1)I_{n+1}(\lambda b)K_{n+1}(\lambda)a_{n}e_{n+1}(w).
\end{equation}
Gathering \eqref{DG at 0}, \eqref{DGis at 0}, \eqref{Df1N1 at 0}, \eqref{Df1N2 at 0}, \eqref{Df1O1 at 0}, \eqref{Df2N1 at 0}, \eqref{Df2N2 at 0} and \eqref{Df2O2 at 0}, we get the desired result. The proof of Proposition \ref{proposition linearized operator} is now complete.
\end{proof}
\subsection{Asymptotic monotonicity of the eigenvalues}\label{sec asym vp}
This subsection is devoted to the proof of Proposition \ref{proposition monotonicity of the eigenvalues} concerning the asymptotic monotonicity of the eigenvalues needed to ensure the one dimensional kernel assumption of Crandall-Rabinowitz's Theorem. But first, we have to prove their existence and this is the purpose of the following lemma.
\begin{lem}\label{lemma existence of eigenvalues}
Let $\lambda>0$ and $b\in(0,1).$ There exists $N_{0}(\lambda,b)\in\mathbb{N}^{*}$ such that for all integer $n\geqslant N_{0}(\lambda,b),$ there exist two angular velocities 
\begin{align}\label{definition af the angular velocities}
\Omega_{n}^{\pm}(\lambda,b)&:=\frac{1-b^{2}}{2b}\Lambda_{1}(\lambda,b)+\frac{1}{2}\Big(\Omega_{n}(\lambda)-\Omega_{n}(\lambda b)\Big)\nonumber\\
&\quad\pm\frac{1}{2b}\sqrt{\Big(b\big[\Omega_{n}(\lambda)+\Omega_{n}(\lambda b)\big]-(1+b^{2})\Lambda_{1}(\lambda,b)\Big)^{2}-4b^{2}\Lambda_{n}^{2}(\lambda,b)}
\end{align}
for which the matrix $M_{n}\left(\lambda,b,\Omega_{n}^{\pm}(\lambda,b)\right)$ is singular.
\end{lem}
\begin{proof}
The determinant of $M_{n}(\lambda,b,\Omega)$ is 
\begin{align}\label{expression of the determinant}
\det\big(M_{n}(\lambda,b,\Omega)\big) &  =  \Big(\Omega_{n}(\lambda)-\Omega-b\Lambda_{1}(\lambda,b)\Big)\Big(\Lambda_{1}(\lambda,b)-b\big[\Omega_{n}(\lambda b)+\Omega\big]\Big)+b\Lambda_{n}^{2}(\lambda,b)\nonumber\\
& =  b\Omega^{2}-B_{n}(\lambda,b)\Omega+C_{n}(\lambda,b),
\end{align}
where
\begin{align*}
	B_{n}(\lambda,b)&:=(1-b^{2})\Lambda_{1}(\lambda,b)+b\big[\Omega_{n}(\lambda)-\Omega_{n}(\lambda b)\big],\\
	C_{n}(\lambda,b)&:=b\left[\left(\Lambda_{1}(\lambda,b)-\frac{1}{b}\Omega_{n}(\lambda)\right)\Big(b\Omega_{n}(\lambda b)-\Lambda_{1}(\lambda,b)\Big)+\Lambda_{n}^{2}(\lambda,b)\right].
\end{align*}
It is a polynomial of degree two in $\Omega$ which has at most two roots. Let us compute its discriminant. After straightforward computations, we find 
\begin{align}\label{def deltan}
\Delta_{n}(\lambda,b) & :=  B_{n}^{2}(\lambda,b)-4bC_{n}(\lambda,b)\nonumber\\
& = \Big(b\big[\Omega_{n}(\lambda)+\Omega_{n}(\lambda b)\big]-(1+b^{2})\Lambda_{1}(\lambda,b)\Big)^{2}-4b^{2}\Lambda_{n}^{2}(\lambda,b).
\end{align}
Using the asymptotic expansion of large order \eqref{asymptotic expansion of high order}, we infer 
\begin{equation}\label{asymp-1}
	\forall\lambda>0,\quad \forall b\in(0,1],\quad I_{n}(\lambda b)K_{n}(\lambda)\underset{n\rightarrow\infty}{\longrightarrow}0.
\end{equation}
As a consequence, 
\begin{equation}\label{convergence of Deltan towards Deltainfty}
\Delta_{n}(\lambda,b)\underset{n\rightarrow\infty}{\longrightarrow}\Delta_{\infty}(\lambda,b),
\end{equation}
where
\begin{equation}\label{definition of deltainfty}
\Delta_{\infty}(\lambda,b)=\delta_{\infty}^{2}(\lambda,b)\quad\mbox{ with }\quad\delta_{\infty}(\lambda,b):=b\big[I_{1}(\lambda)K_{1}(\lambda)+I_{1}(\lambda b)K_{1}(\lambda b)\big]-(1+b^{2})I_{1}(\lambda b)K_{1}(\lambda).
\end{equation}
We can rewrite $\delta_{\infty}(\lambda,b)$ as 
$$\delta_{\infty}(\lambda,b)=\big[bI_{1}(\lambda)-I_{1}(\lambda b)\big]K_{1}(\lambda)+bI_{1}(\lambda b)\big[K_{1}(\lambda b)-bK_{1}(\lambda)\big].$$
According to \eqref{ratio bounds with derivatives} and \eqref{symmetry Bessel}, we find $K_1'<0$ on $(0,\infty)$, which implies in turn the strict decay property of $K_{1}$ on $(0,\infty)$. Therefore, since $b\in(0,1),$ we get 
$$bK_{1}(\lambda)<K_{1}(\lambda)<K_{1}(\lambda b).$$
Now since $b\in(0,1)$, we obtain from \eqref{definition of modified Bessel function of first kind},
$$I_1(\lambda b)=\sum_{m=0}^{\infty}\frac{\left(\frac{\lambda b}{2}\right)^{1+2m}}{m!\Gamma(m+2)}< b\sum_{m=0}^{\infty}\frac{\left(\frac{\lambda}{2}\right)^{1+2m}}{m!\Gamma(m+2)}=bI_1(\lambda).$$ 
Finally, 
$$\Delta_{\infty}(\lambda,b)>0.$$
\noindent Thus 
\begin{equation}\label{criterion for Deltan to be positive}
\exists N_{0}(\lambda,b)\in\mathbb{N}^{*},\quad\forall n\in\mathbb{N}^{*},\quad n\geqslant N_{0}(\lambda,b)\Rightarrow\Delta_{n}(\lambda,b)>0.
\end{equation}
Therefore, for $n\geqslant N_{0}(\lambda,b)$ there exist two angular velocities $\Omega_{n}^{-}(\lambda,b)$ and $\Omega_{n}^{+}(\lambda,b)$ for which the matrix $M_{n}(\lambda,b,\Omega_{n}^{\pm}(\lambda,b))$ is singular. These angular velocities are defined by 
\begin{align*}
\Omega_{n}^{\pm}(\lambda,b)&:=\frac{ B_{n}(\lambda,b)\pm\sqrt{\Delta_{n}(\lambda,b)}}{2b}\\
&=\frac{1-b^{2}}{2b}\Lambda_{1}(\lambda,b)+\frac{1}{2}\Big(\Omega_{n}(\lambda)-\Omega_{n}(\lambda b)\Big)\\
&\quad\pm\frac{1}{2b}\sqrt{\Big(b\big[\Omega_{n}(\lambda)+\Omega_{n}(\lambda b)\big]-(1+b^{2})\Lambda_{1}(\lambda,b)\Big)^{2}-4b^{2}\Lambda_{n}^{2}(\lambda,b)}.
\end{align*}
This ends the proof of Lemma \ref{lemma existence of eigenvalues}. 
\end{proof}

We shall now study the monotonicity of the eigenvalues obtained in Lemma \ref{lemma existence of eigenvalues}. This is a crucial point to obtain later the one dimensional condition for the kernel of the linearized operator given by Proposition \ref{proposition linearized operator}. 
\begin{prop}\label{proposition monotonicity of the eigenvalues}
Let $\lambda>0$ and $b\in(0,1).$ There exists $N(\lambda,b)\in\mathbb{N}^{*}$ with $N(\lambda,b)\geqslant N_{0}(\lambda,b)$ where $N_{0}(\lambda,b)$ is defined in Lemma \ref{lemma existence of eigenvalues} such that 
\begin{enumerate}[label=(\roman*)]
\item The sequence $\left(\Omega_{n}^{+}(\lambda,b)\right)_{n\geqslant N(\lambda,b)}$ is strictly increasing and converges to $\Omega_{\infty}^{+}(\lambda,b)=I_{1}(\lambda)K_{1}(\lambda)-b\Lambda_{1}(\lambda,b).$
\item The sequence $\left(\Omega_{n}^{-}(\lambda,b)\right)_{n\geqslant N(\lambda,b)}$ is strictly decreasing and converges to $\Omega_{\infty}^{-}(\lambda,b)=\frac{\Lambda_{1}(\lambda,b)}{b}-I_{1}(\lambda b)K_{1}(\lambda b).$
\end{enumerate}
Then, we have for all $(m,n)\in(\mathbb{N}^{*})^{2}$ with $N(\lambda,b)\leqslant n<m$,
$$\Omega_{\infty}^{-}(\lambda,b)<\Omega_{m}^{-}(\lambda,b)<\Omega_{n}^{-}(\lambda,b)<\Omega_{n}^{+}(\lambda,b)<\Omega_{m}^{+}(\lambda,b)<\Omega_{\infty}^{+}(\lambda,b).$$
\end{prop}
\begin{proof}
The convergence is an immediate consequence of \eqref{definition af the angular velocities}, \eqref{convergence of Deltan towards Deltainfty}, \eqref{definition of deltainfty} and \eqref{asymp-1}. Then, we turn to the asymptotic monotonicity. For that purpose, we study the sign of the difference
$$\Omega_{n+1}^{\pm}(\lambda,b)-\Omega_{n}^{\pm}(\lambda,b)=\frac{1}{2}\Big(\big[\Omega_{n+1}(\lambda)-\Omega_{n}(\lambda)\big]-\big[\Omega_{n+1}(\lambda b)-\Omega_{n}(\lambda b)\big]\Big)\pm\frac{1}{2b}\left[\sqrt{\Delta_{n+1}(\lambda,b)}-\sqrt{\Delta_{n}(\lambda,b)}\right]$$
for $n$ large enough.\\
$\blacktriangleright$ We first study the difference term before the square roots. We can write 
\begin{align*}
&\big[\Omega_{n+1}(\lambda)-\Omega_{n+1}(\lambda b)\big]-\big[\Omega_{n}(\lambda)-\Omega_{n}(\lambda b)\big]\\
&=\big[\Omega_{n+1}(\lambda)-\Omega_{n}(\lambda)\big]-\big[\Omega_{n+1}(\lambda b)-\Omega_{n}(\lambda b)\big]\\
&=\big[I_{n}(\lambda)K_{n}(\lambda)-I_{n+1}(\lambda)K_{n+1}(\lambda)\big]-\big[I_{n}(\lambda b)K_{n}(\lambda b)-I_{n+1}(\lambda b)K_{n+1}(\lambda b)\big]\\
&:=\varphi_{n}(\lambda)-\varphi_{n}(\lambda b).
\end{align*}
By vitue of \eqref{asymptotic expansion of high order for the product InKn}, we deduce 
$$I_{n}(\lambda)K_{n}(\lambda)\underset{n\rightarrow\infty}{=}\frac{1}{2n}-\frac{\lambda^{2}}{4n^{3}}+o_{\lambda}\left(\frac{1}{n^{4}}\right).$$
Therefore, 
\begin{align*}
\varphi_{n}(\lambda)-\varphi_{n}(\lambda b)&\underset{n\rightarrow\infty}{=}\lambda^{2}(b^{2}-1)\frac{(n+1)^{3}-n^{3}}{4n^{3}(n+1)^{3}}+o_{\lambda,b}\left(\frac{1}{n^{4}}\right)\\
&\underset{n\rightarrow\infty}{=}\frac{3\lambda^{2}(b^{2}-1)}{4n^{4}}+o_{\lambda,b}\left(\frac{1}{n^{4}}\right).
\end{align*}
We conclude that
\begin{equation}\label{Asymptotic expansion for the difference term befor the square root}
\frac{1}{2}\Big(\big[\Omega_{n+1}(\lambda)-\Omega_{n}(\lambda)\big]-\big[\Omega_{n+1}(\lambda b)-\Omega_{n}(\lambda b)\big]\Big)\underset{n\rightarrow\infty}{=}O_{\lambda,b}\left(\frac{1}{n^{4}}\right).
\end{equation}
$\blacktriangleright$ The next task is to look at the asymptotic sign of the difference $\sqrt{\Delta_{n+1}(\lambda,b)}-\sqrt{\Delta_{n}(\lambda,b)}.$
We can write 
$$\sqrt{\Delta_{n+1}(\lambda,b)}-\sqrt{\Delta_{n}(\lambda,b)}=\frac{\Delta_{n+1}(\lambda,b)-\Delta_{n}(\lambda,b)}{\sqrt{\Delta_{n+1}(\lambda,b)}+\sqrt{\Delta_{n}(\lambda,b)}}$$
with 
\begin{align*}
\Delta_{n+1}(\lambda,b)-\Delta_{n}(\lambda,b)&=b\Big(\Omega_{n+1}(\lambda)-\Omega_{n}(\lambda)+\Omega_{n+1}(\lambda b)-\Omega_{n}(\lambda b)\Big)\\
&\quad\times\Big(b\big[\Omega_{n+1}(\lambda)+\Omega_{n}(\lambda)+\Omega_{n+1}(\lambda b)+\Omega_{n}(\lambda b)\big]-2(1+b^{2})\Lambda_{1}(\lambda,b)\Big)\\
&\quad+4b^{2}\Big(\Lambda_{n}(\lambda,b)-\Lambda_{n+1}(\lambda,b)\Big)\Big(\Lambda_{n}(\lambda,b)+\Lambda_{n+1}(\lambda,b)\Big).
\end{align*}
By using \eqref{asymptotic expansion of high order for the product InKn}, we have 
$$\Lambda_{n}(\lambda,b)\underset{n\rightarrow\infty}{=}\frac{b^{n}}{2n}+\frac{\lambda^{2}b^{n}(b^{2}-1)}{2n^{2}}+o_{\lambda,b}\left(\frac{b^{n}}{n^{2}}\right).$$
Hence, the following asymptotic expansion holds 
$$\Lambda_{n}(\lambda,b)\pm\Lambda_{n+1}(\lambda,b)\underset{n\rightarrow\infty}{=}o_{\lambda,b}\left(\frac{1}{n^{2}}\right).$$
As a consequence,
\begin{equation}\label{asymptotic expansion 2}
4b^{2}\Big(\Lambda_{n}(\lambda,b)-\Lambda_{n+1}(\lambda,b)\Big)\Big(\Lambda_{n}(\lambda,b)+\Lambda_{n+1}(\lambda,b)\Big)\underset{n\rightarrow\infty}{=}o_{\lambda,b}\left(\frac{1}{n^{2}}\right).
\end{equation}
In addition,
\begin{equation}\label{asymptotic behaviour 3}
b\Big(\Omega_{n+1}(\lambda)-\Omega_{n}(\lambda)+\Omega_{n+1}(\lambda b)-\Omega_{n}(\lambda b)\Big)=b\Big(\varphi_{n}(\lambda)+\varphi_{n}(\lambda b)\Big)\underset{n\rightarrow\infty}{\sim}\frac{b}{n^{2}}
\end{equation}
and
\begin{equation}\label{asymptotic behaviour 4}
\begin{array}{l}
b\Big[\Omega_{n+1}(\lambda)+\Omega_{n}(\lambda)+\Omega_{n+1}(\lambda b)+\Omega_{n}(\lambda b)\Big]-2(1+b^{2})\Lambda_{1}(\lambda,b)\\
=2b\Big[I_{1}(\lambda)K_{1}(\lambda)+I_{1}(\lambda b)K_{1}(\lambda b)\Big]-2(1+b^{2})I_{1}(\lambda b)K_{1}(\lambda)\\
\mbox{\hspace{1cm}}-b\Big[I_{n+1}(\lambda)K_{n+1}(\lambda)+I_{n+1}(\lambda b)K_{n+1}(\lambda b)+I_{n}(\lambda)K_{n}(\lambda)+I_{n}(\lambda b)K_{n}(\lambda b)\Big]\\
\underset{n\rightarrow\infty}{\longrightarrow}2\delta_{\infty}(\lambda,b),
\end{array}
\end{equation}
where $\delta_{\infty}(\lambda,b)$ is defined in \eqref{definition of deltainfty}. From \eqref{convergence of Deltan towards Deltainfty}, \eqref{definition of deltainfty}, \eqref{asymptotic expansion 2}, \eqref{asymptotic behaviour 3} and \eqref{asymptotic behaviour 4}, we obtain
\begin{equation}\label{asymptotic expansion for difference square root}
\sqrt{\Delta_{n+1}(\lambda,b)}-\sqrt{\Delta_{n}(\lambda,b)}\underset{n\rightarrow\infty}{\sim}\frac{b}{n^{2}}.
\end{equation}
$\blacktriangleright$ Combining \eqref{Asymptotic expansion for the difference term befor the square root} and \eqref{asymptotic expansion for difference square root}, we get
$$\Omega_{n+1}^{\pm}(\lambda,b)-\Omega_{n}^{\pm}(\lambda,b)\underset{n\rightarrow\infty}{\sim}\pm\frac{1}{2n^{2}}.$$
We conclude that there exists $N(\lambda,b)\geqslant N_{0}(\lambda,b)$ such that
$$\forall n\in\mathbb{N}^{*},\quad n\geqslant N(\lambda,b)\Rightarrow\left\lbrace\begin{array}{l}
\Omega_{n+1}^{+}(\lambda,b)-\Omega_{n}^{+}(\lambda,b)>0\\
\Omega_{n+1}^{-}(\lambda,b)-\Omega_{n}^{-}(\lambda,b)<0,
\end{array}\right.$$
i.e. the sequence $\left(\Omega_{n}^{+}(\lambda,b)\right)_{n\geqslant N(\lambda,b)}$ $\Big($resp. $\left(\Omega_{n}^{-}(\lambda,b)\right)_{n\geqslant N(\lambda,b)}\Big)$ is strictly increasing (resp. decreasing). This achieves the proof of Proposition \ref{proposition monotonicity of the eigenvalues}.
\end{proof}
We shall now study both important asymptotic behaviours 
$$\lambda\rightarrow 0\quad\mbox{and}\quad b\rightarrow 0.$$
The first one corresponds to the Euler case and the second one corresponds to the simply-connected case. We remark that we formally recover (at least partially) \cite[Thm. B.]{HHMV16} and \cite[Thm. 5.1.]{DHR19} looking at these limits. More precisely, we have the following result.
\begin{lem}\label{lem continuity spectrum}
	The spectrum is continuous in the following sense.
	\begin{enumerate}[label=(\roman*)]
		\item Let $b\in(0,1).$ There exists $\widetilde{N}(b)$ such that 
		$$\forall n\in\mathbb{N}^*,\,n\geqslant\widetilde{N}(b)\,\Rightarrow\,\Omega_{n}^{\pm}(\lambda,b)\underset{\lambda\rightarrow0}{\longrightarrow}\Omega_{n}^{\pm}(b),$$
		where $\Omega_{n}^{\pm}(b)$ is defined in \eqref{def eig Euler}.
		\item Let $\lambda>0.$ There exists $\widetilde{N}(\lambda)$ such that 
		$$\forall n\in\mathbb{N}^*,\,n\geqslant\widetilde{N}(\lambda)\,\Rightarrow\,\Omega_{n}^{+}(\lambda,b)\underset{b\rightarrow0}{\longrightarrow}\Omega_{n}(\lambda),$$ 
		where $\Omega_{n}(\lambda)$ is defined in \eqref{def eig sc QGSW}.
	\end{enumerate}
\end{lem}
\begin{proof}
\textbf{(i)} In view of \eqref{asymptotic expansion of small argument}, we deduce 
\begin{equation}\label{lim in lbd of Lbd}
	\forall\, n\in\mathbb{N}^{*},\quad \forall\, b\in(0,1],\quad I_{n}(\lambda b)K_{n}(\lambda)\underset{\lambda\rightarrow 0}{\longrightarrow}\frac{b^{n}}{2n}.
\end{equation}
In what follows, we fix $b\in(0,1).$ By virtue of \eqref{lim in lbd of Lbd}, the matrices $M_{n}$ defined in Proposition \ref{proposition linearized operator}, satisfy the following convergence
$$\forall n\in\mathbb{N}^*,\quad M_{n}(\lambda,b,\Omega)\underset{\lambda\rightarrow 0}{\longrightarrow}M_{n}(b,\Omega):=\left(\begin{array}{cc}
	\frac{n-1}{2n}-\frac{b^{2}}{2}-\Omega & \frac{b^{n+1}}{2n}\\
	-\frac{b^{n}}{2n} & \frac{b}{2}-\frac{b(n-1)}{2n}-b\Omega 
\end{array}\right).$$
After straightforward computations, we find 
$$\det\big(M_{n}(b,\Omega)\big)=b\Omega^{2}-\frac{b(1-b^{2})}{2}\Omega+\frac{b}{4n^{2}}\left[n(1-b^{2})-1+b^{2n}\right].$$
This polynomial of degree two in $\Omega$ has the  discriminant 
$$\Delta_{n}(b):=\frac{b^{2}}{n^{2}}\left[\left(\frac{n(1-b^{2})}{2}-1\right)^{2}-b^{2n}\right].$$
Thus, provided $\Delta_{n}(b)>0$, i.e. for 
\begin{equation}\label{cond-edc}
	1+b^{n}-\frac{n(1-b^{2})}{2}<0,
\end{equation}
we have two roots 
$$\Omega_{n}^{\pm}(b):=\frac{1-b^{2}}{4}\pm\frac{1}{2n}\sqrt{\left(\frac{n(1-b^{2})}{2}-1\right)^{2}-b^{2n}}.$$
Then, we recover the result found in \cite[Thm. B.]{HHMV16}. Now, observe that the sequence $n\mapsto 1+b^{n}-\frac{n(1-b^{2})}{2}$ is decreasing. Then there exists $\widetilde{N}(b)\in\mathbb{N}^*$ and $c_0>0$ such that
$$\inf_{n\in\mathbb{N}^*\atop n\geqslant\widetilde{N}(b)}\Delta_{n}(b)\geqslant c_0>0.$$ 
We use the integral representation \eqref{form-Lebe0}, allowing  to write
$$\forall n\in\mathbb{N}^*,\quad I_n(\lambda)K_n(\lambda)-\tfrac{1}{2n}=\tfrac{1}{2}\int_0^\infty\Big[J_0\left(2\lambda\sinh\big(\tfrac{t}{2}\big)\right)-1\Big]e^{-nt}dt.$$
Now using the integral representation \eqref{Besse-repr}, we find
$$J_0\left(2\lambda\sinh\big(\tfrac{t}{2}\big)\right)-1=\tfrac{1}{\pi}\int_0^\pi\Big[\cos\big(2\lambda\sinh\left(\tfrac{t}{2}\right)\sin(\theta)\big)-1\Big]d\theta.$$
The classical inequalities
$$\forall x\in\mathbb{R},\quad |\cos(x)-1|\leqslant\tfrac{x^2}{2}\quad\textnormal{and}\quad\sinh(x)\leqslant\tfrac{e^x}{2}$$
provide the following estimate for $t\geqslant0$
$$\Big|J_0\left(2\lambda\sinh\big(\tfrac{t}{2}\big)\right)-1\Big|\leqslant \lambda^2e^{t}.$$
We conclude that
\begin{equation}\label{unif n cv InKn}
	\forall\lambda>0,\quad\sup_{n\in\mathbb{N}\setminus\{0,1\}}\big|I_n(\lambda)K_n(\lambda)-\tfrac{1}{2n}\big|\leqslant\lambda^2.
\end{equation}
On the other hand, we set for $\varepsilon>0$,
$$K_{0}^{\varepsilon}(x)=K_0(\varepsilon x)+\log\left(\tfrac{\varepsilon}{2}\right).$$
Remark that \eqref{explicit form for K0} implies
$$\lim_{\varepsilon\rightarrow 0}K_{0}^{\varepsilon}(x)=-\log\left(\tfrac{x}{2}\right)-\boldsymbol{\gamma}.$$
By the dominated convergence theorem, one has
$$\forall n\in\mathbb{N}^*,\quad\lim_{\varepsilon\rightarrow 0}\int_{\mathbb{T}}K_{0}^{\varepsilon}(|1-be^{\ii\theta}|)\cos(n\theta)d\eta=-\int_{\mathbb{T}}\log(|1-be^{\ii\theta}|)\cos(n\theta)d\theta.$$
Now one obtains from \eqref{simpl int}
\begin{align*}
	\forall n\in\mathbb{N}^*,\quad\int_{\mathbb{T}}K_{0}^{\varepsilon}(|1-be^{\ii\theta}|)\cos(n\theta)d\eta&=\int_{\mathbb{T}}K_{0}(\varepsilon|1-be^{\ii\theta}|)\cos(n\theta)d\theta\\
	&=I_{n}(\varepsilon b)K_{n}(\varepsilon).
\end{align*}
Putting together the last two equality with \eqref{lim in lbd of Lbd} yields
$$\forall n\in\mathbb{N}^*,\quad\int_{\mathbb{T}}\log\big(|1-be^{\ii\theta}|\big)d\theta=-\tfrac{b^n}{2n}.$$
Added to \eqref{simpl int}, we have
$$\forall \lambda>0,\quad\forall n\in\mathbb{N}^*,\quad I_n(\lambda b)K_n(\lambda)-\tfrac{b^{n}}{2n}=\int_{\mathbb{T}}\Big[K_0\big(\lambda|1-be^{\ii\theta}|\big)+\log\big(|1-be^{\ii\theta}|\big)\Big]\cos(n\theta)d\theta.$$
Then, making appeal to the power series decompositions \eqref{explicit form for K0} and \eqref{definition of modified Bessel function of first kind}, we get
\begin{equation}\label{unif n cv Lbd}
	\forall \lambda>0,\quad\sup_{n\in\mathbb{N}^*}\big|I_n(\lambda b)K_n(\lambda)-\tfrac{b^{n}}{2n}\big|\lesssim\max(|\log(\lambda)|,1)\lambda^2.
\end{equation}
Combining \eqref{def deltan}, \eqref{unif n cv InKn}, \eqref{unif n cv Lbd} and \eqref{lim in lbd of Lbd} one obtains 
$$\sup_{n\in\mathbb{N}^*}\big|\Delta_{n}(\lambda,b)-\Delta_{n}(b)\big|\underset{\lambda\rightarrow 0}{\longrightarrow}0.$$
Hence, there exists $\lambda_0(b)>0$ such that
$$\inf_{\lambda\in(0,\lambda_0(b)]}\inf_{n\in\mathbb{N}^*\atop n\geqslant\widetilde{N}(b)}\Delta_{n}(\lambda,b)\geqslant \tfrac{c_0}{2}>0.$$
Therefore, we deduce from \eqref{definition af the angular velocities} and \eqref{lim in lbd of Lbd} that,
$$\forall n\in\mathbb{N}^*,\quad n\geqslant\widetilde{N}(b)\,\Rightarrow\,\Omega_{n}^{\pm}(\lambda,b)\underset{\lambda\rightarrow 0}{\longrightarrow}\Omega_{n}^{\pm}(b).$$

\noindent \textbf{(ii)} In what follows, we fix $\lambda>0.$ By using the asymptotic \eqref{asymptotic expansion of small argument}, we find 
$$\frac{\Lambda_{1}(\lambda,b)}{b}\underset{b\rightarrow 0}{\longrightarrow}\frac{\lambda K_{1}(\lambda)}{2}\quad\textnormal{and}\quad\forall n\in\mathbb{N}^*,\,\Lambda_{n}(\lambda,b)\underset{b\rightarrow 0}{\sim}\tfrac{\left(\lambda b\right)^{n}}{2^nn!}K_n(\lambda).$$
Using the power series decomposition \eqref{definition of modified Bessel function of first kind}, the decay property of $\lambda\mapsto I_n(\lambda)K_n(\lambda)$ and the asymptotic \eqref{lim in lbd of Lbd}, we get
$$\forall n\in\mathbb{N}^*,\quad\big|I_{n}(\lambda b)K_n(\lambda)-\tfrac{\left(\lambda b\right)^{n}}{2^nn!}K_n(\lambda)\big|\leqslant b^2I_n(\lambda)K_n(\lambda)\leqslant b^2.$$
Thus, we obtain from \eqref{def deltan}, \eqref{unif n cv InKn} and \eqref{lim in lbd of Lbd}
\begin{equation}\label{equiv disc b0}
	\sup_{n\in\mathbb{N}^*}\Big|\Delta_{n}(\lambda,b)-b^2\Big[\big(\Omega_{n}(\lambda)+\tfrac{n-1}{2n}-\tfrac{\lambda K_{1}(\lambda)}{2}\big)^2-\tfrac{\left(\lambda b\right)^{2n}}{2^{2n}(n!)^2}K_n^2(\lambda)\Big]\Big|\underset{b\to 0}{\longrightarrow}0.
\end{equation}
Notice that
$$\Omega_{n}(\lambda)+\tfrac{n-1}{2n}-\tfrac{\lambda K_{1}(\lambda)}{2}\underset{n\rightarrow\infty}{\longrightarrow}I_1(\lambda)K_1(\lambda)+\tfrac{1-\lambda K_{1}(\lambda)}{2}.$$
Consider the function $\varphi$ defined by $\forall x>0,\varphi(x)=x K_{1}(x).$ From \eqref{Bessel derivatives}, we get 
$$\varphi'(x)=K_{1}(x)+x K_{1}'(x)=-x K_{0}(x)<0.$$
Hence $\varphi$ is strictly decreasing on $(0,\infty).$ Moreover, in view of the asymptotic \eqref{asymptotic expansion of small argument}, we infer 
$$\lim_{x\rightarrow 0}\varphi(x)=1.$$
Thus, using also \eqref{symmetry Bessel}, we obtain
$$\forall x>0,\quad\varphi(x)\in(0,1).$$
Therefore, we deduce that there exists $\widetilde{N}(\lambda)\in\mathbb{N}^*$ such that
$$\forall n\in\mathbb{N}^*,\quad n\geqslant\widetilde{N}(\lambda)\,\Rightarrow\,\Omega_{n}(\lambda)+\tfrac{n-1}{2n}-\tfrac{\lambda K_{1}(\lambda)}{2}>0.$$
In addition, using \eqref{asymptotic expansion of high order} and up to increase the value of $\widetilde{N}(\lambda)$ one gets
$$\forall n\in\mathbb{N}^*,\quad n\geqslant\widetilde{N}(\lambda)\,\Rightarrow\,\tfrac{\left(\lambda b\right)^{2n}}{2^{2n}(n!)^2}K_n^2(\lambda)\leqslant 1.$$
Coming back to \eqref{equiv disc b0}, we infer the existence of $b_0(\lambda)\in(0,1)$ such that
$$\forall b\in(0,b_0(\lambda)),\quad\forall n\in\mathbb{N}^*,\quad n\geqslant\widetilde{N}(\lambda)\,\Rightarrow\,\Delta_{n}(\lambda,b)>0.$$
Thus, we get from \eqref{definition af the angular velocities}
$$\forall n\in\mathbb{N}^*,\quad n\geqslant\widetilde{N}(\lambda)\,\Rightarrow\,\Omega_{n}^{+}(\lambda,b)\underset{b\rightarrow 0}{\longrightarrow}\Omega_{n}(\lambda).$$
Then, we partially recover the result found in \cite[Thm. 5.1.]{DHR19}. We also obtain, up to increase the value of $\widetilde{N}(\lambda)$,
$$\forall n\in\mathbb{N}^*,\quad n\geqslant\widetilde{N}(\lambda)\,\Rightarrow\,\Omega_{n}^{-}(\lambda,b)\underset{b\rightarrow 0}{\longrightarrow}\Omega_{n}^{-}(\lambda):=\frac{\lambda nK_{1}(\lambda)-n+1}{2n}.$$
Unfortunately, we cannot prove bifurcation from these eigenvalues.
\end{proof}

\section{Bifurcation from simple eigenvalues}
We prove here the following result which implies the main Theorem \ref{main theorem} by a direct application of Crandall-Rabinowitz's Theorem \ref{Crandall-Rabinowitz theorem}. 
\begin{prop}\label{proposition assumptions of CR theorem}
Let $\lambda>0,$ $b\in(0,1),$ $\alpha\in(0,1)$ and $\mathbf{m}\in\mathbb{N}^{*}$ such that $\mathbf{m}\geqslant N(\lambda,b).$ Then the following assertions hold true.
\begin{enumerate}[label=(\roman*)]
\item There exists $r>0$ such that $G(\lambda,b,\cdot,\cdot,\cdot):\mathbb{R}\times B_{r,\mathbf{m}}^{1+\alpha}\times B_{r,\mathbf{m}}^{1+\alpha}\rightarrow Y_{\mathbf{m}}^{\alpha}$ is well-defined and of class $C^{1}.$
\item The kernel $\ker\Big(DG\big(\lambda,b,\Omega_{\mathbf{m}}^{\pm}(\lambda,b),0,0\big)\Big)$ is one-dimensional and generated by 
$$v_{0,\mathbf{m}}:\begin{array}[t]{rcl}
\mathbb{T} & \rightarrow & \mathbb{C}^{2}\\
w & \mapsto & \left(\begin{array}{c}
b\big[\Omega_{\mathbf{m}}(\lambda b)+\Omega_{\mathbf{m}}^{\pm}(\lambda,b)\big]-\Lambda_{1}(\lambda,b)\\
-\Lambda_{\mathbf{m}}(\lambda,b)
\end{array}\right)\overline{w}^{\mathbf{m}-1}.
\end{array}$$
\item The range $R\Big(DG\big(\lambda,b,\Omega_{\mathbf{m}}^{\pm}(\lambda,b),0,0\big)\Big)$ is closed and of codimension one in $Y_{\mathbf{m}}^{\alpha}.$
\item Tranversality condition :
$$\partial_{\Omega}DG\big(\lambda,b,\Omega_{\mathbf{m}}^{\pm}(\lambda,b),0,0\big)(v_{0,\mathbf{m}})\not\in R\Big(DG\big(\lambda,b,\Omega_{\mathbf{m}}^{\pm}(\lambda,b),0,0\big)\Big).$$
\end{enumerate}
\end{prop}
\begin{proof}
\textbf{(i)} Follows from Proposition \ref{proposition regularity of the functional}.\\
\textbf{(ii)} Let $(h_{1},h_{2})\in X_{\mathbf{m}}^{1+\alpha}.$  We write
\begin{equation}\label{form of h1 and h2 with symmetries}
h_{1}(w)=\sum_{n=1}^{\infty}a_{n}\overline{w}^{n\mathbf{m}-1}\quad\mbox{ and }\quad h_{2}(w)=\sum_{n=1}^{\infty}b_{n}\overline{w}^{n\mathbf{m}-1}.
\end{equation}
Proposition \ref{proposition linearized operator} gives 
\begin{equation}\label{form of DG with symmetries}
\forall w\in\mathbb{T},\quad DG(\lambda,b,\Omega,0,0)(h_{1},h_{2})(w)=\sum_{n=1}^{\infty}n\mathbf{m}M_{n\mathbf{m}}(\lambda,b,\Omega)\left(\begin{array}{c}
a_{n}\\
b_{n}
\end{array}\right)e_{n\mathbf{m}}(w).
\end{equation}
For $\Omega\in\left\lbrace\Omega_{\mathbf{m}}^{-}(\lambda,b),\Omega_{\mathbf{m}}^{+}(\lambda,b)\right\rbrace,$ we have 
$$\det\Big(M_{\mathbf{m}}\big(\lambda,b,\Omega_{\mathbf{m}}^{\pm}(\lambda,b)\big)\Big)=0.$$
Thus, the kernel of $DG\big(\lambda,\Omega_{\mathbf{m}}^{\pm}(\lambda,b),0,0\big)$ is non trivial and it is one dimensional if and only if 
\begin{equation}\label{non zero determinant}
\forall n\in\mathbb{N}^{*},\quad n\geqslant 2\Rightarrow\det\Big(M_{n\mathbf{m}}\big(\lambda,b,\Omega_{\mathbf{m}}^{\pm}(\lambda,b)\big)\Big)\neq 0.
\end{equation}
The previous condition is satisfied in view of Proposition \ref{proposition monotonicity of the eigenvalues}.
Hence, we have the equivalence 
$$(h_{1},h_{2})\in\ker\Big(DG\big(\lambda,b,\Omega_{\mathbf{m}}^{\pm}(\lambda,b),0,0\big)\Big)\Leftrightarrow\left\lbrace\begin{array}{l}
\forall n\in\mathbb{N}^{*},\quad n\geqslant 2\Rightarrow a_{n}=0=b_{n}\\
\left(\begin{array}{c}
a_{1}\\
b_{1}
\end{array}\right)\in\ker\Big(M_{\mathbf{m}}\big(\lambda,b,\Omega_{\mathbf{m}}^{\pm}(\lambda,b)\big)\Big).
\end{array}\right.$$
Therefore, we can select as generator of $\ker\Big(DG\big(\lambda,b,\Omega_{\mathbf{m}}^{\pm}(\lambda,b),0,0\big)\Big)$ the following pair of functions 
$$v_{0,\mathbf{m}}:\begin{array}[t]{rcl}
\mathbb{T} & \rightarrow & \mathbb{C}^{2}\\
w & \mapsto & \left(\begin{array}{c}
b\big[\Omega_{\mathbf{m}}(\lambda b)+\Omega_{\mathbf{m}}^{\pm}(\lambda,b)\big]-\Lambda_{1}(\lambda,b)\\
-\Lambda_{\mathbf{m}}(\lambda,b)
\end{array}\right)\overline{w}^{\mathbf{m}-1}.
\end{array}$$
\textbf{(iii)} We consider the set $Z_{\mathbf{m}}$ defined by
\begin{align*}
Z_{\mathbf{m}} & :=  \displaystyle\left\lbrace g=(g_{1},g_{2})\in Y_{\mathbf{m}}^{\alpha}\quad\textnormal{s.t.}\quad \forall w\in\mathbb{T},\quad g(w)=\sum_{n=1}^{\infty}\left(\begin{array}{c}
\mathscr{A}_{n}\\
\mathscr{B}_{n}
\end{array}\right)e_{n\mathbf{m}}(w),\right.\\
&  \mbox{\hspace{1cm}}\forall n\in\mathbb{N}^{*},\quad\left.(\mathscr{A}_{n},\mathscr{B}_{n})\in\mathbb{R}^{2}\quad\mbox{and}\quad\exists(a_{1},b_{1})\in\mathbb{R}^{2},\quad M_{\mathbf{m}}\big(\lambda,b,\Omega_{\mathbf{m}}^{\pm}(\lambda,b)\big)\left(\begin{array}{c}
a_{1}\\
b_{1}
\end{array}\right)=\left(\begin{array}{c}
\mathscr{A}_{1}\\
\mathscr{B}_{1}
\end{array}\right)\right\rbrace.
\end{align*}
Clearly, $Z_{\mathbf{m}}$ is a closed sub-vector space of codimension one in $Y_{\mathbf{m}}^{\alpha}.$ It remains to prove that it coincides with the range of $DG\big(\lambda,b,\Omega_{\mathbf{m}}^{\pm}(\lambda,b),0,0\big).$ Obviously, we have the inclusion
$$R\Big(DG\big(\lambda,b,\Omega_{\mathbf{m}}^{\pm}(\lambda,b),0,0\big)\Big)\subset Z_{\mathbf{m}}.$$
We are left to prove the converse inclusion. Let $(g_{1},g_{2})\in Z_{\mathbf{m}}.$ We shall prove that the equation 
$$DG\big(\lambda,b,\Omega_{\mathbf{m}}^{\pm}(\lambda,b),0,0\big)(h_{1},h_{2})=(g_{1},g_{2})$$ 
admits a solution $(h_{1},h_{2})\in X_{\mathbf{m}}^{1+\alpha}$ in the form \eqref{form of h1 and h2 with symmetries}. According to \eqref{form of DG with symmetries}, the previous equation is equivalent to the following countable set of equations 
$$\forall n\in\mathbb{N}^{*},\quad n\mathbf{m}M_{n\mathbf{m}}\big(\lambda,b,\Omega_{\mathbf{m}}^{\pm}(\lambda,b)\big)\left(\begin{array}{c}
a_{n}\\
b_{n}
\end{array}\right)=\left(\begin{array}{c}
\mathscr{A}_{n}\\
\mathscr{B}_{n}
\end{array}\right).$$
For $n=1,$ the existence follows from the definition of $Z_{\mathbf{m}}.$ Thanks to \eqref{non zero determinant}, the sequences $(a_{n})_{n\geqslant 2}$ and $(b_{n})_{n\geqslant 2}$ are uniquely determined by
$$\forall n\in\mathbb{N}^{*},\quad n\geqslant 2\Rightarrow\left(\begin{array}{c}
a_{n}\\
b_{n}
\end{array}\right)=\frac{1}{n\mathbf{m}}M_{n\mathbf{m}}^{-1}\big(\lambda,b,\Omega_{\mathbf{m}}^{\pm}(\lambda,b)\big)\left(\begin{array}{c}
\mathscr{A}_{n}\\
\mathscr{B}_{n}
\end{array}\right),$$
or equivalently,
$$\left\lbrace\begin{array}{rcl}
a_{n} & = & \displaystyle\frac{\Lambda_{1}(\lambda,b)-b\big[\Omega_{n\mathbf{m}}(\lambda b)+\Omega_{\mathbf{m}}^{\pm}(\lambda,b)\big]}{n\mathbf{m}\det\Big(M_{n\mathbf{m}}\big(\lambda,b,\Omega_{\mathbf{m}}^{\pm}(\lambda,b)\big)\Big)}\mathscr{A}_{n}-\frac{b\Lambda_{n\mathbf{m}}(\lambda,b)}{n\mathbf{m}\det\Big(M_{n\mathbf{m}}\big(\lambda,b,\Omega_{\mathbf{m}}^{\pm}(\lambda,b)\big)\Big)}\mathscr{B}_{n}\\
& & \\
b_{n} & = & \displaystyle\frac{\Lambda_{n\mathbf{m}}(\lambda,b)}{n\mathbf{m}\det\Big(M_{n\mathbf{m}}\big(\lambda,b,\Omega_{\mathbf{m}}^{\pm}(\lambda,b)\big)\Big)}\mathscr{A}_{n}+\frac{\Omega_{n\mathbf{m}}(\lambda b)+\Omega_{\mathbf{m}}^{\pm}(\lambda,b)-b\Lambda_{1}(\lambda,b)}{n\mathbf{m}\det\Big(M_{n\mathbf{m}}\big(\lambda,b,\Omega_{\mathbf{m}}^{\pm}(\lambda,b)\big)\Big)}\mathscr{B}_{n}.
\end{array}\right.$$
It remains to prove the regularity, that is $(h_{1},h_{2})\in X_{\mathbf{m}}^{1+\alpha}.$ For that purpose, we show 
$$w\mapsto\left(\begin{array}{c}
h_{1}(w)-a_{1}\overline{w}^{\mathbf{m}-1}\\
h_{2}(w)-a_{2}\overline{w}^{\mathbf{m}-1}
\end{array}\right)\in C^{1+\alpha}(\mathbb{T})\times C^{1+\alpha}(\mathbb{T}).$$
We may focus on the first component, the second one being analogous. We set
$$H_{1}(\lambda,b,\mathbf{m})(w):=\sum_{n=2}^{\infty}\frac{\mathscr{A}_{n}}{n\det\Big(M_{n\mathbf{m}}\big(\lambda,b,\Omega_{\mathbf{m}}^{\pm}(\lambda,b)\big)\Big)}w^{n},\quad H_{2}(w):=\sum_{n=2}^{\infty}\frac{\mathscr{B}_{n}}{n}w^{n}$$
and
$$\mathscr{G}_{1}(\lambda,b,\mathbf{m})(w):=\sum_{n=2}^{\infty}I_{n\mathbf{m}}(\lambda b)K_{n\mathbf{m}}(\lambda b)w^{n},\quad \mathscr{G}_{2}(\lambda,b,\mathbf{m})(w):=\sum_{n=2}^{\infty}\frac{\Lambda_{n\mathbf{m}}(\lambda,b)}{\det\Big(M_{n\mathbf{m}}\big(\lambda,b,\Omega_{\mathbf{m}}^{\pm}(\lambda,b)\big)\Big)}w^{n}.$$
If we denote $\widetilde{h}_{1}(w):=h_{1}(w)-a_{1}\overline{w}^{\mathbf{m}-1}$, then we can write 
\begin{align}\label{dec h1t}
\widetilde{h}_{1}(w)= & C_{1}(\lambda,b,\mathbf{m})wH_{1}(\lambda,b,\mathbf{m})\left(\overline{w}^{\mathbf{m}}\right)\nonumber\\
&+C_{2}(b,\mathbf{m})w(\mathscr{G}_{1}(\lambda,b,\mathbf{m})\ast H_{1}(\lambda,b,\mathbf{m}))\left(\overline{w}^{\mathbf{m}}\right)\nonumber\\
&+C_{2}(b,\mathbf{m})w(\mathscr{G}_{2}(\lambda,b,\mathbf{m})\ast H_{2})\left(\overline{w}^{\mathbf{m}}\right),
\end{align}
where
\begin{align*}
C_{1}(\lambda,b,\mathbf{m}) & :=  \displaystyle\frac{\Lambda_{1}(\lambda,b)-b\Omega_{\mathbf{m}}^{\pm}(\lambda,b)-bI_{1}(\lambda b)K_{1}(\lambda b)}{\mathbf{m}},\\
C_{2}(b,\mathbf{m}) & :=  \displaystyle -\frac{b}{\mathbf{m}}.
\end{align*}
The convolution must be understood in the usual sens, that is 
$$\forall w=e^{\ii\theta}\in\mathbb{T},\quad f\ast g(w)=\fint_{\mathbb{T}}f(\tau)g(w\overline{\tau})\frac{d\tau}{\tau}=\frac{1}{2\pi}\int_{0}^{2\pi}f\left(e^{\ii\eta}\right)g\left(e^{\ii(\theta-\eta)}\right)d\eta.$$
We shall use the classical convolution law
\begin{equation}\label{convol}
	L^{1}(\mathbb{T})\ast C^{1+\alpha}(\mathbb{T})\hookrightarrow C^{1+\alpha}(\mathbb{T}).
\end{equation}
By using the decay property of the product $I_{n}K_{n}$ and the asymptotic \eqref{asymptotic expansion of small argument}, we have 
$$\| \mathscr{G}_{1}(\lambda,b,\mathbf{m})\|_{L^{1}(\mathbb{T})}\lesssim\| \mathscr{G}_{1}(\lambda,b,\mathbf{m})\|_{L^{2}(\mathbb{T})}=\left(\sum_{n=2}^{\infty}I_{n\mathbf{m}}^{2}(\lambda b)K_{n\mathbf{m}}^{2}(\lambda b)\right)^{\frac{1}{2}}\leqslant\frac{1}{2\mathbf{m}}\left(\sum_{n=2}^{\infty}\frac{1}{n^{2}}\right)^{\frac{1}{2}}<\infty.$$
We also have 
$$\| \mathscr{G}_{2}(\lambda,b,\mathbf{m})\|_{L^{1}(\mathbb{T})}\leqslant\| \mathscr{G}_{2}(\lambda,b,\mathbf{m})\|_{L^{\infty}(\mathbb{T})}\lesssim \sum_{n=2}^{\infty}b^{n\mathbf{m}}<\infty.$$
Hence
\begin{equation}\label{int G1G2}
	\Big(\mathscr{G}_{1}(\lambda,b,\mathbf{m}),\mathscr{G}_{2}(\lambda,b,\mathbf{m})\Big)\in\left(L^{1}(\mathbb{T})\right)^{2}.
\end{equation}
We now prove that $H_{1}$ and $H_{2}$ are with regularity $C^{1+\alpha}(\mathbb{T}).$\\
$\blacktriangleright$ \textit{Regularity of $H_{2}$ :}\\
First observe that by Cauchy-Schwarz inequality and the embedding $C^{\alpha}(\mathbb{T})(\hookrightarrow L^{\infty}(\mathbb{T}))\hookrightarrow L^{2}(\mathbb{T})$, we have 
\begin{align}\label{est H2}
\| H_{2}\|_{L^{\infty}(\mathbb{T})} & \leqslant  \displaystyle\sum_{n=2}^{\infty}\frac{|\mathscr{B}_{n}|}{n}\nonumber\\
& \leqslant  \displaystyle\left(\sum_{n=2}^{\infty}\frac{1}{n^{2}}\right)^{\frac{1}{2}}\left(\sum_{n=2}^{\infty}\mathscr{|B}_{n}|^{2}\right)^{\frac{1}{2}}\nonumber\\
& \lesssim  \| g_{2}\|_{L^{2}(\mathbb{T})}\nonumber\\
& \lesssim  \| g_{2}\|_{C^{\alpha}(\mathbb{T})}.
\end{align}
We now have to prove that $H_{2}'\in C^{\alpha}(\mathbb{T}).$ We show that it coincides, up to slight modifications, with $g_{2}$ which is of regularity $C^{\alpha}(\mathbb{T}).$ For that purpose, we show that we can differentiate $H_{2}$ term by term.\\
We denote $(S_{N})_{N\geqslant 2}$ (resp. $(R_{N})_{N\geqslant 2}$) the sequence of the partial sums (resp. the sequence of the remainders) of the series of functions $H_{2}.$ One has
$$R_{N}(w)=\sum_{n=N+1}^{\infty}\frac{\mathscr{B}_{n}}{n}w^{n}.$$
Using Cauchy-Schwarz inequality, we obtain similarly to \eqref{est H2}
\begin{align*}
	\|R_{N}\|_{L^{\infty}(\mathbb{T})}\leqslant\left(\sum_{n=N+1}^{\infty}\frac{1}{n^{2}}\right)^{\frac{1}{2}}\|g_{2}\|_{C^{\alpha}(\mathbb{T})}\underset{N\rightarrow\infty}{\longrightarrow}0.
\end{align*}
Hence
\begin{equation}\label{CVU SN}
	\|S_{N}-H_{2}\|_{L^{\infty}(\mathbb{T})}\underset{N\rightarrow\infty}{\longrightarrow}0.
\end{equation}
One has
$$S_{N}'(w)=\overline{w}\sum_{n=2}^{N}\mathscr{B}_{n}w^{n}:=\overline{w}g_{2}^{N}(w).$$
We set 
$$g_{2}^{+}(w):=\sum_{n=2}^{\infty}\mathscr{B}_{n}w^{n}.$$
By continuity of the Szegö projection defined by
$$\Pi:\sum_{n\in\mathbb{Z}}\alpha_{n}w^{n}\mapsto\sum_{n\in\mathbb{N}}\alpha_{n}w^{n}$$
from $C^{\alpha}(\mathbb{T})$ into itself (see \cite{HH15} for more details) added to the fact that $g_{2}\in C^{\alpha}(\mathbb{T})$, we deduce that $g_{2}^{+}\in C^{\alpha}(\mathbb{T}).$ Applying Bernstein Theorem of Fourier series gives that $g_{2}^{+}$ is the uniform limit of its Fourier series, namely
\begin{equation}\label{CVU SN'}
	\| S_{N}'- \overline{w}g_{2}^{+}\|_{L^{\infty}(\mathbb{T})}\underset{N\rightarrow\infty}{\longrightarrow}0.
\end{equation}
Gathering \eqref{CVU SN} and \eqref{CVU SN'}, we conclude that we can differentiate $H_{2}$ term by term and get
$$H_{2}'(\omega)=\overline{w}g_{2}^{+}(w).$$
As a consequence,
\begin{equation}\label{reg H2}
	H_{2}\in C^{1+\alpha}(\mathbb{T}).
\end{equation}
$\blacktriangleright$ \textit{Regularity of $H_{1}(\lambda,b,\mathbf{m})$ :}\\
By using \eqref{expression of the determinant} and \eqref{asymptotic expansion of high order for the product InKn}, we have the asymptotic expansion
\begin{equation}\label{asymptotic expansion of det(Mnm)}
\det\left(M_{n\mathbf{m}}(\lambda,b,\Omega_{\mathbf{m}}^{\pm}(\lambda,b))\right)\underset{n\rightarrow \infty}{=}d_{\infty}(\lambda,b,\mathbf{m})+\frac{\widetilde{d}_{\infty}(\lambda,b,\mathbf{m})}{n}+O_{\lambda,b,\mathbf{m}}\left(\frac{1}{n^{3}}\right),
\end{equation}
with, using Proposition \ref{proposition monotonicity of the eigenvalues},
\begin{align*}
	d_{\infty}(\lambda,b,\mathbf{m})&:=\left[I_{1}(\lambda)K_{1}(\lambda)-\Omega_{\mathbf{m}}^{\pm}(\lambda,b)-b\Lambda_{1}(\lambda,b)\right]\left[\Lambda_{1}(\lambda,b)-b\Omega_{\mathbf{m}}^{\pm}(\lambda,b)-bI_{1}(\lambda b)K_{1}(\lambda b)\right]\\
	&=b\left[\Omega_{\infty}^{+}(\lambda,b)-\Omega_{\mathbf{m}}^{\pm}(\lambda,b)\right]\left[\Omega_{\infty}^{-}(\lambda,b)-\Omega_{\mathbf{m}}^{\pm}(\lambda,b)\right]\\
	&<0
\end{align*}
and, using \eqref{definition of deltainfty},
\begin{align*}
\widetilde{d}_{\infty}(\lambda,b,\mathbf{m}) & :=  \displaystyle\frac{b}{2\mathbf{m}}\left[I_{1}(\lambda)K_{1}(\lambda)-\Omega_{\mathbf{m}}^{\pm}(\lambda,b)-b\Lambda_{1}(\lambda,b)\right]-\frac{1}{2\mathbf{m}}\left[\Lambda_{1}(\lambda,b)-b\Omega_{\mathbf{m}}^{\pm}(\lambda,b)-bI_{1}(\lambda b)K_{1}(\lambda b)\right]\\
& =  \displaystyle\frac{b\left(I_{1}(\lambda)K_{1}(\lambda)+I_{1}(\lambda b)K_{1}(\lambda b)\right)-(1+b^{2})\Lambda_{1}(\lambda,b)}{2\mathbf{m}}\\
& =  \displaystyle\frac{\delta_{\infty}(\lambda,b)}{2\mathbf{m}}.
\end{align*}
We denote
\begin{equation}\label{def rn}
	r_{n}(\lambda,b,\mathbf{m}):=\det\left(M_{n\mathbf{m}}(\lambda,b,\Omega_{\mathbf{m}}^{\pm}(\lambda,b))\right)-d_{\infty}(\lambda,b,\mathbf{m})\underset{n\rightarrow\infty}{=}\frac{\widetilde{d}_{\infty}(\lambda,b,\mathbf{m})}{n}+O_{\lambda,b,\mathbf{m}}\left(\frac{1}{n^{3}}\right).
\end{equation}
We can write
$$\frac{1}{\det\Big(M_{n\mathbf{m}}\big(\lambda,b,\Omega_{\mathbf{m}}^{\pm}(\lambda,b)\big)\Big)}=\frac{r_{n}^{2}(\lambda,b,\mathbf{m})}{d_{\infty}^{2}(\lambda,b,\mathbf{m})\det\Big(M_{n\mathbf{m}}\big(\lambda,b,\Omega_{\mathbf{m}}^{\pm}(\lambda,b)\big)\Big)}-\frac{r_{n}(\lambda,b,\mathbf{m})}{d_{\infty}^{2}(\lambda,b,\mathbf{m})}+\frac{1}{d_{\infty}(\lambda,b,\mathbf{m})}.$$
Thus we can write
\begin{align}\label{dec H1}
H_{1}(\lambda,b,\mathbf{m})(w)&= \displaystyle\frac{1}{d_{\infty}^{2}(\lambda,b,\mathbf{m})}\sum_{n=2}^{\infty}\frac{\mathscr{A}_{n}r_{n}^{2}(\lambda,b,\mathbf{m})}{n\det\Big(M_{n\mathbf{m}}\big(\lambda,b,\Omega_{\mathbf{m}}^{\pm}(\lambda,b)\big)\Big)}w^{n}-\frac{1}{d_{\infty}^{2}(\lambda,b,\mathbf{m})}\sum_{n=2}^{\infty}\frac{\mathscr{A}_{n}r_{n}(\lambda,b,\mathbf{m})}{n}w^{n}\nonumber\\
&\quad+\frac{1}{d_{\infty}(\lambda,b,\mathbf{m})}\sum_{n=2}^{\infty}\frac{\mathscr{A}_{n}}{n}w^{n}\nonumber\\
&:=\displaystyle \frac{1}{d_{\infty}^{2}(\lambda,b,\mathbf{m})}H_{1,1}(\lambda,b,\mathbf{m})(w)-\frac{1}{d_{\infty}^{2}(\lambda,b,\mathbf{m})}H_{1,2}(\lambda,b,\mathbf{m})(w)+\frac{1}{d_{\infty}(\lambda,b,\mathbf{m})}H_{1,3}(\lambda,b,\mathbf{m})(w).
\end{align}
Now since $(\mathscr{A}_{n})_{n\in\mathbb{N}^{*}}\in l^{2}(\mathbb{N}^{*})\subset l^{\infty}(\mathbb{N}^{*}),$ we have 
$$\left|\frac{\mathscr{A}_{n}r_{n}^{2}(\lambda,b,\mathbf{m})}{n\det\Big(M_{n\mathbf{m}}\big(\lambda,b,\Omega_{\mathbf{m}}^{\pm}(\lambda,b)\big)\Big)}\right|\underset{n\rightarrow\infty}{=}O_{\lambda,b,\mathbf{m}}\left(\frac{1}{n^{3}}\right).$$
By using the link regularity/decay of Fourier coefficients, we deduce that 
\begin{equation}\label{reg H11}
	H_{1,1}(\lambda,b,\mathbf{m})\in C^{1+\alpha}(\mathbb{T}).
\end{equation}
Similarly to \eqref{reg H2}, we can obtain
\begin{equation}\label{reg H13}
	H_{1,3}(\lambda,b,\mathbf{m})\in C^{1+\alpha}(\mathbb{T}).
\end{equation}
By the same method, we can also differentiate term by term $H_{1,2}(\lambda,b,\mathbf{m})$ and obtain
$$\forall w\in\mathbb{T},\quad \big(H_{1,2}(\lambda,b,\mathbf{m})\big)'(w)=\overline{w}\sum_{n=2}^{\infty}\mathscr{A}_{n}r_{n}(\lambda,b,\mathbf{m})w^{n}.$$
Notice that from \eqref{def rn}, we can write
$$\forall w\in\mathbb{T},\quad w\big(H_{1,2}(\lambda,b,\mathbf{m})\big)'(w)=\widetilde{d}_{\infty}(\lambda,b,\mathbf{m})H_{1,3}(\lambda,b,\mathbf{m})+(\mathscr{C}\ast g_{1}^{+})(w),$$
where
\begin{align*}
	\forall w\in\mathbb{T},\quad g_{1}^{+}(w):=\sum_{n=2}^{\infty}\mathscr{A}_{n}w^{n}\quad\textnormal{and}\quad 
	\mathscr{C}(w):=\sum_{n=2}^{\infty}\mathscr{C}_{n}w^{n}\quad\textnormal{with}\quad \mathscr{C}_{n}=O_{\lambda,b,\mathbf{m}}\left(\frac{1}{n^{3}}\right).
\end{align*}
Using again the continuity of the Szegö projection, we have 
\begin{equation}\label{reg g1 C}
	g_{1}^{+}\in C^{1+\alpha}(\mathbb{T})\subset L^{\infty}(\mathbb{T})\subset L^{1}(\mathbb{T})\quad\textnormal{and}\quad \mathscr{C}\in C^{1+\alpha}(\mathbb{T}).
	\end{equation}
Using \eqref{reg H13}, \eqref{reg g1 C} and \eqref{convol}, we deduce that
$$\big(H_{1,2}(\lambda,b,\mathbf{m})\big)'\in C^{1+\alpha}(\mathbb{T})\subset C^{\alpha}(\mathbb{T}).$$ 
Thus 
\begin{equation}\label{reg H12}
	H_{1,2}(\lambda,b,\mathbf{m})\in C^{1+\alpha}(\mathbb{T}).
\end{equation}
Gathering \eqref{reg H11}, \eqref{reg H12} and \eqref{reg H13}, we conclude that
\begin{equation}\label{reg H1}
	H_{1}(\lambda,b,\mathbf{m})\in C^{1+\alpha}(\mathbb{T}).
\end{equation}
Putting together \eqref{dec h1t}, \eqref{reg H1}, \eqref{reg H2}, \eqref{int G1G2} and \eqref{convol}, we finally conclude
$$\widetilde{h}_{1}\in C^{1+\alpha}(\mathbb{T}).$$
\noindent\textbf{(iv)} $\Omega_{\mathbf{m}}^{\pm}(\lambda,b)$ is a simple eigenvalue since $\Delta_{\mathbf{m}}(\lambda,b)>0.$ From \eqref{split GON} and \eqref{DOj}, we deduce
$$\left\lbrace\begin{array}{l}
\partial_{\Omega}DG_{1}\big(\lambda,b,\Omega_{\mathbf{m}}^{\pm}(\lambda,b),0,0\big)(h_{1},h_{2})(w)=\mbox{Im}\left\lbrace\overline{h_{1}'(w)}+\overline{w}h_{1}(w)\right\rbrace=-\displaystyle\sum_{n=0}^{\infty}n\mathbf{m}a_{n}e_{n\mathbf{m}}(w)\\
\partial_{\Omega}DG_{2}\big(\lambda,b,\Omega_{\mathbf{m}}^{\pm}(\lambda,b),0,0\big)(h_{1},h_{2})(w)=b\mbox{Im}\left\lbrace\overline{h_{2}'(w)}+\overline{w}h_{2}(w)\right\rbrace=-\displaystyle\sum_{n=0}^{\infty}bn\mathbf{m}b_{n}e_{n\mathbf{m}}(w).
\end{array}\right.$$
Thus,
$$\partial_{\Omega}DG\big(\lambda,b,\Omega_{\mathbf{m}}^{\pm}(\lambda,b),0,0\big)(v_{0,\mathbf{m}})(w)=\mathbf{m}\left(\begin{array}{c}
\Lambda_{1}(\lambda,b)-b\big[\Omega_{\mathbf{m}}(\lambda b)+\Omega_{\mathbf{m}}^{\pm}(\lambda,b)\big]\\
b\Lambda_{\mathbf{m}}(\lambda,b)
\end{array}\right)e_{\mathbf{m}}(w).$$
Notice that the previous expression belongs to the range of $DG\big(\lambda,b,\Omega_{\mathbf{m}}^{\pm}(\lambda,b),0,0\big)$ if and only if the vector  
$$\left(\begin{array}{c}
\Lambda_{1}(\lambda,b)-b\big[\Omega_{\mathbf{m}}(\lambda b)+\Omega_{\mathbf{m}}^{\pm}(\lambda,b)\big]\\
b\Lambda_{\mathbf{m}}(\lambda,b)
\end{array}\right)$$
is a scalar multiple of one column of the matrix $M_{\mathbf{m}}\big(\lambda,b,\Omega_{\mathbf{m}}^{\pm}(\lambda,b)\big).$ This occurs if and only if
\begin{equation}\label{condition for partialOmegaDG to be in the range}
\Big(\Lambda_{1}(\lambda,b)-b\big[\Omega_{\mathbf{m}}(\lambda b)+\Omega_{\mathbf{m}}^{\pm}(\lambda,b)\big]\Big)^{2}-b^{2}\Lambda_{\mathbf{m}}^{2}(\lambda,b)=0.
\end{equation}
Putting \eqref{condition for partialOmegaDG to be in the range} together with $\det\Big(M_{\mathbf{m}}\big(\lambda,b,\Omega_{\mathbf{m}}^{\pm}(\lambda,b)\big)\Big)=0$ implies 
$$\Big(\Lambda_{1}(\lambda,b)-b\big[\Omega_{\mathbf{m}}(\lambda b)+\Omega_{\mathbf{m}}^{\pm}(\lambda,b)\big]\Big)\Big((1-b^{2})\Lambda_{1}(\lambda,b)+b\big[\Omega_{\mathbf{m}}(\lambda)-\Omega_{\mathbf{m}}(\lambda b)\big]-2b\Omega_{\mathbf{m}}^{\pm}(\lambda,b)\Big)=0.$$
Now remark that the above equation is equivalent to
$$\Lambda_{1}(\lambda,b)-b\big[\Omega_{\mathbf{m}}(\lambda b)+\Omega_{\mathbf{m}}^{\pm}(\lambda,b)\big]=0\quad\mbox{or}\quad\Omega_{\mathbf{m}}^{\pm}(\lambda,b)=\frac{1}{2b}\Big((1-b^{2})\Lambda_{1}(\lambda,b)+b\big[\Omega_{\mathbf{m}}(\lambda)-\Omega_{\mathbf{m}}(\lambda b)\big]\Big).$$
Since $b\neq 0$ and $\Lambda_{\mathbf{m}}(\lambda,b)\neq 0$, then in view of \eqref{condition for partialOmegaDG to be in the range}, the first equation can't be solved. Then, necessary, the second equation must be satisfied. But we notice that it corresponds to a multiple eigenvalue ($\Delta_{\mathbf{m}}(\lambda,b)=0$), which is excluded here. Therefore, we conclude that 
$$\partial_{\Omega}DG\big(\lambda,b,\Omega_{\mathbf{m}}^{\pm}(\lambda,b),0,0\big)(v_{0,\mathbf{m}})\not\in R\Big(DG\big(\lambda,b,\Omega_{\mathbf{m}}^{\pm}(\lambda,b),0,0\big)\Big).$$
This ends the proof of Proposition \ref{proposition assumptions of CR theorem}.
\end{proof}

\appendix
\section{Formulae on modified Bessel functions}\label{appendix Bessel}
We shall collect some useful  information on  modified Bessel functions. For more details we refer to \cite{AS64,W95}.
We define first the Bessel functions of order $\nu\in\mathbb{C}$ by 
$$J_{\nu}(z)=\sum_{m=0}^{\infty}\frac{(-1)^{m}\left(\frac{z}{2}\right)^{\nu +2m}}{m!\Gamma(\nu+m+1)},\quad|\mbox{arg}(z)|<\pi.$$
Notice that when $\nu\in\mathbb{N}$ we have the following integral representation, see \cite[p. 115]{Leb65}.
\begin{equation}\label{Besse-repr}
	J_\nu(x)=\frac{1}{\pi}\int_0^\pi\cos\big(x \sin \theta-\nu \theta\big) d\theta.
\end{equation}
We define the Bessel functions of imaginary argument by
\begin{equation}\label{definition of modified Bessel function of first kind}
I_{\nu}(z)=\sum_{m=0}^{\infty}\frac{\left(\frac{z}{2}\right)^{\nu+2m}}{m!\Gamma(\nu+m+1)},\mbox{ }\quad|\mbox{arg}(z)|<\pi
\end{equation}
and $$K_{\nu}(z)=\frac{\pi}{2}\frac{I_{-\nu}(z)-I_{\nu}(z)}{\sin(\nu\pi)},\mbox{ }\quad\nu\in\mathbb{C}\backslash\mathbb{Z},\mbox{ }\quad|\mbox{arg}(z)|<\pi.$$
For $n\in\mathbb{Z},$ we define $K_{n}(z)=\displaystyle\lim_{\nu\rightarrow n}K_{\nu}(z).$
We give now useful properties of modified Bessel functions.\\ 
\textbf{Symmetry and positivity properties} (see \cite[p. 375]{AS64}) \textbf{:}
\begin{equation}\label{symmetry Bessel}
\forall n\in\mathbb{N},\quad\forall\lambda>0,\quad I_{-n}(\lambda)=I_{n}(\lambda)>0\quad\mbox{ and }\quad K_{-n}(\lambda)=K_{n}(\lambda)>0.
\end{equation}
\textbf{Derivatives} (see \cite[p. 376]{AS64}) \textbf{:}\\
If we set $\mathcal{Z}_{\nu}(z)=I_{\nu}(z)$ or $e^{i\nu\pi}K_{\nu}(z)$, then for all $\nu\in\mathbb{R}$, we have 
\begin{equation}\label{Bessel derivatives}
	\mathcal{Z}_{\nu}'(z)=\mathcal{Z}_{\nu-1}(z)-\frac{\nu}{z}\mathcal{Z}_{\nu}(z)=\mathcal{Z}_{\nu+1}(z)+\frac{\nu}{z}\mathcal{Z}_{\nu}(z).
\end{equation}
\textbf{Power series extension for $K_{n}$} (see \cite[p. 375]{AS64}) \textbf{:}\\
\begin{align*}
	K_{n}(z)=&\frac{1}{2}\left(\frac{z}{2}\right)^{-n}\sum_{k=0}^{n-1}\frac{(n-k-1)!}{k!}\left(\frac{-z}{4}\right)^{k}+(-1)^{n+1}\ln\left(\frac{z}{2}\right)I_{n}(z)\\
	&+\frac{1}{2}\left(\frac{-z}{2}\right)^{n}\sum_{k=0}^{\infty}\left(\psi(k+1)+\psi(n+k+1)\right)\frac{\left(\frac{z^{2}}{4}\right)^{k}}{k!(n+k)!},
\end{align*}
where 
$$\psi(1)=-\boldsymbol{\gamma} \textnormal{ (Euler's constant)}\quad \textnormal{and}\quad  \forall m\in\mathbb{N}^{*},\,\,\psi(m+1)=\displaystyle\sum_{k=1}^{m}\frac{1}{k}-\boldsymbol{\gamma}.$$
In particular 
\begin{equation}\label{explicit form for K0}
K_{0}(z)=-\log\left(\frac{z}{2}\right)I_{0}(z)+\sum_{m=0}^{\infty}\frac{\left(\frac{z}{2}\right)^{2m}}{(m!)^{2}}\psi(m+1),
\end{equation}
so $K_{0}$ behaves like a logarithm at $0.$\\

\noindent\textbf{Decay property for the product $I_{\nu}K_{\nu}$} (see \cite{B09} and \cite{DHR19}) \textbf{:}\\ The application $(\lambda,\nu)\mapsto I_{\nu}(\lambda)K_{\nu}(\lambda)$ is strictly decreasing in each variable $(\lambda,\nu)\in(\mathbb{R}_{+}^{*})^{2}.$\\

\noindent\textbf{Beltrami's summation formula} (see \cite[p. 361]{W95}) \textbf{:}
Let $0<b<a.$ Then
\begin{equation}\label{Beltrami's summation formula}
\forall \theta\in\mathbb{R},\quad K_{0}\left(\sqrt{a^{2}+b^{2}-2ab\cos(\theta)}\right)=\sum_{m=-\infty}^{\infty}I_{m}(b)K_{m}(a)\cos(m\theta).
\end{equation}

\textbf{Ratio bounds} (see \cite{BP13}) \textbf{:}\\
For all $n\in\mathbb{N},$ for all $\lambda\in\mathbb{R}_{+}^{*},$ we have 
\begin{equation}\label{ratio bounds with derivatives}
\left\lbrace\begin{array}{l}
\displaystyle\frac{\lambda I_{n}'(\lambda)}{I_{n}(\lambda)}<\sqrt{\lambda^{2}+n^{2}}\\
\\
\displaystyle\frac{\lambda K_{n}'(\lambda)}{K_{n}(\lambda)}<-\sqrt{\lambda^{2}+n^{2}}
\end{array}\right.
\end{equation}
\textbf{Integral representation for the product} $I_nK_n$ (see \cite[p. 140]{Leb65}) \textbf{:}
\begin{align}\label{form-Lebe0}
	\forall n\in\mathbb{N}^*,\quad\forall\lambda>0,\quad (I_{n}K_{n})(\lambda)=\frac12\int_{0}^\infty J_0\big(2\lambda\sinh(\tfrac{t}{2})\big)e^{-nt}dt.
\end{align}
\textbf{Asymptotic expension of small argument} (see \cite[p. 375]{AS64}) \textbf{:}
\begin{equation}\label{asymptotic expansion of small argument}
\forall n\in\mathbb{N}^{*},\quad I_{n}(\lambda)\underset{\lambda\rightarrow 0}{\sim}\frac{\left(\frac{1}{2}\lambda\right)^{n}}{\Gamma(n+1)}\quad\mbox{ and }\quad K_{n}(\lambda)\underset{\lambda\rightarrow 0}{\sim}\frac{\Gamma(n)}{2\left(\frac{1}{2}\lambda\right)^{n}}.
\end{equation} 
\textbf{Asymptotic expansion of high order} (see \cite[p. 377]{AS64}) \textbf{:}
\begin{equation}\label{asymptotic expansion of high order}
\forall\lambda>0,\quad\mbox{ }I_{\nu}(\lambda)\underset{\nu\rightarrow\infty}{\sim}\frac{1}{\sqrt{2\pi\nu}}\left(\frac{e\lambda}{2\nu}\right)^{\nu}\quad\mbox{ and }\quad K_{\nu}(\lambda)\underset{\nu\rightarrow\infty}{\sim}\sqrt{\frac{\pi}{2\nu}}\left(\frac{e\lambda}{2\nu}\right)^{-\nu}.
\end{equation}
\textbf{Asymptotic expansion of high order for the product $I_{j}K_{j}$} (see \cite{HS11}) \textbf{:}
\begin{equation}\label{asymptotic expansion of high order for the product InKn}
\forall\lambda>0,\quad\forall b\in(0,1],\quad I_{n}(\lambda b)K_{n}(\lambda)\underset{n\rightarrow\infty}{\sim}\frac{b^{n}}{2n}\left(\sum_{m=0}^{\infty}\frac{b_{m}(\lambda b)}{n^{m}}\right)\left(\sum_{m=0}^{\infty}(-1)^{m}\frac{b_{m}(\lambda)}{n^{m}}\right),
\end{equation}
where for each $m\in\mathbb{N}$, $b_{m}(\lambda)$ is a polynomial of degree $m$ in $\lambda^{2}$ defined by 
$$b_{0}(\lambda)=1\quad\mbox{ and }\quad\forall m\in\mathbb{N}^{*},\,\,b_{m}(\lambda)=\sum_{k=1}^{m}(-1)^{m-k}\frac{S(m,k)}{k!}\left(\frac{\lambda^{2}}{4}\right)^{k}$$
and the $S(m,k)$ are Stirling numbers of second kind defined recursively by 
$$\forall (m,k)\in(\mathbb{N}^{*})^{2},\quad S(m,k)=S(m-1,k-1)+kS(m-1,k),$$
with $$S(0,0)=1,\quad\forall m\in\mathbb{N}^{*},\,\,S(m,1)=1\quad\mbox{ and }\quad S(m,0)=0\quad\mbox{ and if }m<k\mbox{ then }S(m,k)=0.$$
\section{Proof of Proposition \ref{proposition regularity of the functional}}\label{appendix proof reg funct}
In this appendix, we prove the regularity result stated in Proposition \ref{proposition regularity of the functional}. The techniques involved are now classical and the following proof follows closely the lines of the proof of \cite[Prop. 4.1]{HHH18}.
\begin{proof}
	\textbf{(i)} The proof proceeds in three steps. The first step is to show the well-posedness of the function $G(\lambda,b,\cdot,\cdot,\cdot)$ from $\mathbb{R}\times B_{r}^{1+\alpha}\times B_{r}^{1+\alpha}$ to $Y^{\alpha}$ for some $r$ small enough. Then, in the second step, we shall prove the existence and give the computation of the Gâteaux derivative of $G(\lambda,b,\cdot,\cdot,\cdot).$ Finally, in the third step, we shall prove that these Gâteaux derivatives are continuous. This will show the $C^{1}$ regularity of $G(\lambda,b,\cdot,\cdot,\cdot).$\\
	
	\noindent $\blacktriangleright$ \textbf{Step 1 : Show that $G(\lambda,b,\cdot,\cdot,\cdot):\mathbb{R}\times B_{r}^{1+\alpha}\times B_{r}^{1+\alpha}\rightarrow Y^{\alpha}$ is well-defined :}\\
	For this purpose, we split $G_{j}$ into two terms, the self-induced term $\mathcal{S}_{j}$ and the interaction term $\mathcal{I}_{j}$, 
	\begin{equation}\label{split GON}
		G_{j}(\lambda,b,\Omega,f_{1},f_{2})=\mathcal{S}_{j}(\lambda,b,\Omega,f_{j})+\mathcal{I}_{j}(\lambda,b,f_{1},f_{2}),
	\end{equation}
	where
	\begin{align*}
		\mathcal{S}_{j}(\lambda,b,\Omega,f_{j})(w)&:=\mbox{Im}\left\lbrace\left[\Omega\Phi_{j}(w)+(-1)^{j}S(\lambda,\Phi_{j},\Phi_{j})(w)\right]\overline{w}\overline{\Phi_{j}'(w)}\right\rbrace,\\
		\mathcal{I}_{j}(\lambda,b,f_{1},f_{2})&:=(-1)^{j-1}\mbox{Im}\left\lbrace S(\lambda,\Phi_{i},\Phi_{j})(w)\overline{w}\overline{\Phi_{j}'(w)}\right\rbrace.
	\end{align*}
	\ding{226} We refer to \cite[Prop. 5.7]{DHR19} for the study of $\mathcal{S}_{j}$. Only the $(-1)^{j}$ defers, but has no consequence. We recall here the results. There exists $r\in(0,1)$ such that for all $\alpha\in(0,1),$ we have 
	\begin{enumerate}[label=\textbullet]
		\item $\mathcal{S}_{j}(\lambda,b,\cdot,\cdot):\mathbb{R}\times B_{r}^{1+\alpha}\rightarrow Y_{1}^{\alpha}$ is of class $C^{1}.$
		\item The restriction $\mathcal{S}_{j}(\lambda,b,\cdot,\cdot):\mathbb{R}\times B_{r,\mathbf{m}}^{1+\alpha}\rightarrow Y_{\mathbf{m}}^{\alpha}$ is well-defined.
	\end{enumerate}
	Moreover, we have 
	\begin{align}\label{DOj}
		D_{f_{j}}&\mathcal{S}_{j}(\lambda,b,\Omega,f_{j})h_{j}(w)  = \Omega\mbox{Im}\left\lbrace h_{j}(w)\overline{w}\overline{\Phi_{j}'(w)}+\Phi_{j}(w)\overline{w}\overline{h_{j}'(w)}\right\rbrace\nonumber\\
		&+(-1)^{j}\mbox{Im}\left\lbrace S(\lambda,\Phi_{j},\Phi_{j})(w)\overline{w}\overline{h_{j}'(w)}+\overline{w}\overline{\Phi_{j}'(w)}\left[A_{1}(\lambda,\Phi_{j},h_{j})(w)+B_{1}(\lambda,\Phi_{j},h_{j})(w)\right]\right\rbrace,
	\end{align}
	where
	\begin{align*}
		A_{1}(\lambda,\Phi_{j},h_{j})(w)&:=\fint_{\mathbb{T}}h_{j}'(\tau)K_{0}\left(\lambda|\Phi_{j}(w)-\Phi_{j}(\tau)|\right)d\tau,\\
		B_{1}(\lambda,\Phi_{j},h_{j})(w)&:=\lambda\fint_{\mathbb{T}}\Phi_{j}'(\tau)K_{0}'\left(\lambda|\Phi_{j}(w)-\Phi_{j}(\tau)|\right)\frac{\mbox{Re}\left(\left(\overline{h_{j}(w)}-\overline{h_{j}(\tau)}\right)\left(\Phi_{j}(w)-\Phi_{j}(\tau)\right)\right)}{|\Phi_{j}(w)-\Phi_{j}(\tau)|}d\tau.
	\end{align*}
	Actually, this is the most difficult part of this proof since in this case, the integrals appearing have singular kernel and the proof uses some results about singular kernels. As we shall see in the remaining of the proof, the terms concerning $\mathcal{I}_{j}$ are not singular.\\
	\ding{226} We shall first show that for $(f_{1},f_{2})\in B_{r}^{1+\alpha}\times B_{r}^{1+\alpha}$, we have $\mathcal{I}_{j}(\lambda,b,f_{1},f_{2})\in C^{\alpha}(\mathbb{T}).$ According to the algebra structure of $C^{\alpha}(\mathbb{T})$, it suffices to show that for $i\neq j,$  $S(\lambda,\Phi_{i},\Phi_{j})\in C^{\alpha}(\mathbb{T}).$
	For that purpose, we consider the operator $\mathcal{T}$ defined by 
	$$\forall w\in\mathbb{T},\quad\mathcal{T}\chi(w):=\fint_{\mathbb{T}}\chi(\tau)K_{0}\left(\lambda|\Phi_{j}(w)-\Phi_{i}(\tau)|\right)d\tau.$$
	But for $w,\tau\in\mathbb{T},$ we have taking $f_{1}$ and $f_{2}$ small functions, 
	$$|\Phi_{1}(w)-\Phi_{2}(\tau)|\leqslant |w-b\tau|+|f_{1}(w)|+|f_{2}(\tau)|\leqslant(1+b)+\| f_{1}\|_{L^{\infty}(\mathbb{T})}+\| f_{2}\|_{L^{\infty}(\mathbb{T})}\leqslant 2(1+b)$$
	and
	$$|\Phi_{1}(w)-\Phi_{2}(\tau)|\geqslant|w-b\tau|-|f_{1}(w)|-|f_{2}(\tau)|\geqslant(1-b)-\| f_{1}\|_{L^{\infty}(\mathbb{T})}-\| f_{2}\|_{L^{\infty}(\mathbb{T})}\geqslant\frac{1-b}{2}.$$
	Since $K_{0}$ is continuous on $\left[\frac{\lambda(1-b)}{2},2\lambda(1+b)\right]$, we have 
	$$\|\mathcal{T}\chi\|_{L^{\infty}(\mathbb{T})}\lesssim\|\chi\|_{L^{\infty}(\mathbb{T})}.$$
	Moreover, taking $w_{1}\neq w_{2}\in\mathbb{T}$, we have by mean value Theorem, since from \eqref{Bessel derivatives} $K_{0}'=-K_{1}$ is continuous on $\left[\frac{\lambda(1-b)}{2},2\lambda(1+b)\right]$, and left triangle inequality 
	\begin{align*}
		\left|\mathcal{T}\chi(w_{1})-\mathcal{T}\chi(w_{2})\right| & \lesssim  \displaystyle\int_{\mathbb{T}}|\chi(\tau)|\left|K_{0}\left(\lambda|\Phi_{j}(w_{1})-\Phi_{i}(\tau)|\right)-K_{0}\left(|\lambda||\Phi_{j}(w_{2})-\Phi_{i}(\tau)|\right)\right||d\tau|\\
		& \lesssim  \|\chi\|_{L^{\infty}(\mathbb{T})}\left|\Phi_{j}(w_{1})-\Phi_{j}(w_{2})\right|.
	\end{align*}
	Using that $\Phi_{j}\in C^{1+\alpha}(\mathbb{T})\hookrightarrow C^{\alpha}(\mathbb{T}),$ we conclude that
	$$\left|\mathcal{T}\chi(w_{1})-\mathcal{T}\chi(w_{2})\right|\lesssim\|\chi\|_{L^{\infty}(\mathbb{T})}\|\Phi_{j}\|_{C^{\alpha}(\mathbb{T})}|w_{1}-w_{2}|^{\alpha}.$$
	We deduce that
	\begin{equation}\label{estimate for mathcalT}
		\|\mathcal{T}\chi\|_{C^{\alpha}(\mathbb{T})}\lesssim\left(1+\|\Phi_{j}\|_{C^{\alpha}(\mathbb{T})}\right)\|\chi\|_{L^{\infty}(\mathbb{T})}.
	\end{equation}
	Applying this with $\chi=\Phi_{j}'$, we find 
	$$\| S(\lambda,\Phi_{i},\Phi_{j})\|_{C^{\alpha}(\mathbb{T})}\lesssim\left(1+\|\Phi_{j}\|_{C^{\alpha}(\mathbb{T})}\right)\|\Phi_{i}'\|_{L^{\infty}(\mathbb{T})}\lesssim\left(1+\|\Phi_{j}\|_{C^{1+\alpha}(\mathbb{T})}\right)\|\Phi_{i}\|_{C^{1+\alpha}(\mathbb{T})}<\infty.$$
	The last point to check is that the Fourier coefficients of $\mathcal{I}_{j}(\lambda,f_{1},f_{2})$ are real. According to the definition of the space $X^{1+\alpha}$, the mapping $\Phi_{j}$ has real coefficients. We deduce that the Fourier coefficients of $\Phi_{j}'$ are also real. Due to the stability of such property under conjugation and multiplication, we only have to prove that the Fourier coefficients of $S(\lambda,\Phi_{i},\Phi_{j})$ are real. This is checked by the following computations. By using \eqref{symmetry Bessel} and the change of variables $\eta\mapsto-\eta$, one has
	\begin{align*}
		\overline{S(\lambda,\Phi_{i},\Phi_{j})(w)} & =  \displaystyle\overline{\fint_{\mathbb{T}}\Phi_{i}'(\tau)K_{0}\left(\lambda|\Phi_{j}(w)-\Phi_{i}(\tau)|\right)d\tau}\\
		& =  \displaystyle\overline{\frac{1}{2\ii\pi}\int_{0}^{2\pi}\Phi_{i}'\left(e^{\ii\eta}\right)K_{0}\left(\lambda|\Phi_{j}(w)-\Phi_{i}\left(e^{\ii\eta}\right))|\right)\ii e^{\ii\eta}d\eta}\\
		&= \displaystyle\frac{1}{2\pi}\int_{0}^{2\pi}\Phi_{i}'\left(e^{-\ii\eta}\right)K_{0}\left(\lambda|\Phi_{j}(\overline{w})-\Phi_{i}\left(e^{-\ii\eta}\right)|\right)e^{-\ii\eta}d\eta\\
		& = \displaystyle\frac{1}{2\ii\pi}\int_{0}^{2\pi}\Phi_{i}'\left(e^{\ii\eta}\right)K_{0}\left(\lambda|\Phi_{j}(\overline{w})-\Phi_{i}\left(e^{\ii\eta}\right)|\right)\ii e^{\ii\eta}d\eta\\
		& =  \displaystyle\fint_{\mathbb{T}}\Phi_{i}'(\tau)K_{0}\left(\lambda|\Phi_{j}(\overline{w})-\Phi_{i}(\tau)|\right)d\tau\\
		& =  S(\lambda,\Phi_{i},\Phi_{j})(\overline{w}). 
	\end{align*}
	
	\noindent $\blacktriangleright$ \textbf{Step 2 : Show the existence and compute the Gâteaux derivatives of $G(\lambda,b,\cdot,\cdot,\cdot)$ :}\\
	\ding{226} The Gâteaux derivative of $\mathcal{I}_{j}$ at $(f_{1},f_{2})$ in the direction $h=(h_{1},h_{2})\in X^{1+\alpha}$ is given by 
	\begin{align}\label{cv DNj}
		D\mathcal{I}_{j}(\lambda,b,f_{1},f_{2})h & =  D_{f_{1}}\mathcal{I}_{j}(\lambda,b,f_{1},f_{2})h_{1}+D_{f_{2}}\mathcal{I}_{j}(\lambda,b,f_{1},f_{2})h_{2}\nonumber\\
		& :=  \displaystyle\lim_{t\rightarrow 0}\frac{1}{t}\left[\mathcal{I}_{j}(\lambda,b,f_{1}+th_{1},f_{2})-\mathcal{I}_{j}(\lambda,b,f_{1},f_{2})\right]\nonumber\\
		&\quad +\lim_{t\rightarrow 0}\frac{1}{t}\left[\mathcal{I}_{j}(\lambda,b,f_{1},f_{2}+th_{2})-\mathcal{I}_{j}(\lambda,b,f_{1},f_{2})\right].
	\end{align}
	The previous limits are understood in the sens of the strong topology of $Y^{\alpha}.$ As a consequence, we need to to prove first the pointwise existence of these limits and then we shall check that these limits exist in the strong topology of $C^{\alpha}(\mathbb{T}).$ To be able to compute the Gâteaux dérivatives, we have to precise that since the beginning of this study we have identified $\mathbb{C}$ with $\mathbb{R}^{2}.$ Hence $\mathbb{C}$ is naturally endowed with the Euclidean scalar product which writes for $z_{1}=a_{1}+\ii b_{1}$ and $z_{2}=a_{2}+\ii b_{2}$ 
	$$\langle z_{1},z_{2}\rangle:=\mbox{Re}(\overline{z_{1}}z_{2})=\frac{1}{2}\left(\overline{z_{1}}z_{2}+z_{1}\overline{z_{2}}\right)=a_{1}a_{2}+b_{1}b_{2}.$$
	By straightforward computations, we infer 
	\begin{align}\label{DfjNj general}
		D_{f_{j}}\mathcal{I}_{j}(\lambda,b,f_{1},f_{2})h_{j}(w) & =  \displaystyle(-1)^{j-1}\mbox{Im}\left\lbrace\overline{w}\overline{h_{j}'(w)}S(\lambda,\Phi_{i},\Phi_{j})(w)\right.\nonumber\\
		& \quad\left.\displaystyle+\frac{\lambda}{2}\overline{w}\overline{\Phi_{j}'(w)}\left(\overline{h_{j}(w)}A(\lambda,\Phi_{i},\Phi_{j})(w)+h_{j}(w)B(\lambda,\Phi_{i},\Phi_{j})(w)\right)\right\rbrace,
	\end{align}
	where
	$$A(\lambda,\Phi_{i},\Phi_{j})(w):=\fint_{\mathbb{T}}\Phi_{i}'(\tau)K_{0}'\left(\lambda|\Phi_{j}(w)-\Phi_{i}(\tau)|\right)\frac{\Phi_{j}(w)-\Phi_{i}(\tau)}{|\Phi_{j}(w)-\Phi_{i}(\tau)|}d\tau:=\fint_{\mathbb{T}}\Phi_{i}'(\tau)K(\lambda,w,\tau)d\tau$$
	and
	$$B(\lambda,\Phi_{i},\Phi_{j})(w):=\fint_{\mathbb{T}}\Phi_{i}'(\tau)K_{0}'\left(\lambda|\Phi_{j}(w)-\Phi_{i}(\tau)|\right)\frac{\Phi_{j}(\overline{w})-\Phi_{i}(\overline{\tau})}{|\Phi_{j}(w)-\Phi_{i}(\tau)|}d\tau=\fint_{\mathbb{T}}\Phi_{i}'(\tau)\overline{K(\lambda,w,\tau)}d\tau.$$
	Since $B$ differs from $A$ only with a conjugation, then, they both satisfy the same estimates in the coming analysis. For all $w\in\mathbb{T},$ we have 
	$$\left|A(\lambda,\Phi_{i},\Phi_{j})(w)\right|\lesssim\int_{\mathbb{T}}|\Phi_{i}'(\tau)|K_{0}\left(\lambda|\Phi_{j}(w)-\Phi_{i}(\tau)|\right)|d\tau|\lesssim\|\Phi_{i}'\|_{L^{\infty}(\mathbb{T})}.$$
	So $$\| A(\lambda,\Phi_{i},\Phi_{j})\|_{L^{\infty}(\mathbb{T})}\lesssim\|\Phi_{i}'\|_{L^{\infty}(\mathbb{T})}.$$
	Let $w_{1}\neq w_{2}\in\mathbb{T}.$ let $\tau\in\mathbb{T}.$ Then
	\begin{align*}
		&\left|K(\lambda,w_{1},\tau)-K(\lambda,w_{2},\tau)\right|\\
		&=\displaystyle\left|K_{0}'\left(\lambda|\Phi_{j}(w_{1})-\Phi_{i}(\tau)|\right)\frac{\Phi_{j}(w_{1})-\Phi_{i}(\tau)}{|\Phi_{j}(w_{1})-\Phi_{i}(\tau)|}-K_{0}'\left(\lambda|\Phi_{j}(w_{2})-\Phi_{i}(\tau)|\right)\frac{\Phi_{j}(w_{2})-\Phi_{i}(\tau)}{|\Phi_{j}(w_{2})-\Phi_{i}(\tau)|}\right|\\
		&\displaystyle\leqslant\left|K_{0}'\left(\lambda|\Phi_{j}(w_{1})-\Phi_{i}(\tau)|\right)-K_{0}'\left(\lambda|\Phi_{j}(w_{2})-\Phi_{i}(\tau)|\right)\right|\\
		&\displaystyle\quad+\left|K_{0}'\left(\lambda|\Phi_{j}(w_{2})-\Phi_{i}(\tau)|\right)\right|\left|\frac{\Phi_{j}(w_{1})-\Phi_{i}(\tau)}{|\Phi_{j}(w_{1})-\Phi_{i}(\tau)|}-\frac{\Phi_{j}(w_{2})-\Phi_{i}(\tau)}{|\Phi_{j}(w_{2})-\Phi_{i}(\tau)|}\right|.
	\end{align*}
	But by right and left triangle inequalities, we get 
	\begin{align*}
		&\left|\frac{\Phi_{j}(w_{1})-\Phi_{i}(\tau)}{|\Phi_{j}(w_{1})-\Phi_{i}(\tau)|}-\frac{\Phi_{j}(w_{2})-\Phi_{i}(\tau)}{|\Phi_{j}(w_{2})-\Phi_{i}(\tau)|}\right|\\
		& =  \left|\frac{\Phi_{j}(w_{1})-\Phi_{j}(w_{2})}{\left|\Phi_{j}(w_{1})-\Phi_{i}(\tau)\right|}+\left(\Phi_{j}(w_{2})-\Phi_{i}(\tau)\right)\left(\frac{1}{\left|\Phi_{j}(w_{1})-\Phi_{i}(\tau)\right|}-\frac{1}{\left|\Phi_{j}(w_{2})-\Phi_{i}(\tau)\right|}\right)\right|\\
		& \leqslant  \frac{\left|\Phi_{j}(w_{1})-\Phi_{j}(w_{2})\right|}{\left|\Phi_{j}(w_{1})-\Phi_{i}(\tau)\right|}+\left|\Phi_{j}(w_{2})-\Phi_{i}(\tau)\right|\frac{\left|\left|\Phi_{j}(w_{2})-\Phi_{i}(\tau)\right|-\left|\Phi_{j}(w_{1})-\Phi_{i}(\tau)\right|\right|}{\left|\Phi_{j}(w_{1})-\Phi_{i}(\tau)\right|\left|\Phi_{j}(w_{2})-\Phi_{i}(\tau)\right|}\\
		& \leqslant  \frac{2\left|\Phi_{j}(w_{1})-\Phi_{j}(w_{2})\right|}{\left|\Phi_{j}(w_{1})-\Phi_{i}(\tau)\right|}\\
		& \lesssim  \left|\Phi_{j}(w_{1})-\Phi_{j}(w_{2})\right|.
	\end{align*}
	Hence,
	$$\left|K(\lambda,w_{1},\tau)-K(\lambda,w_{2},\tau)\right|\lesssim\left|\Phi_{j}(w_{1})-\Phi_{j}(w_{2})\right|\lesssim\|\Phi_{j}\|_{C^{\alpha}(\mathbb{T})}|w_{1}-w_{2}|^{\alpha}.$$
	Thus,
	$$\| A(\lambda,\Phi_{i},\Phi_{j})\|_{C^{\alpha}(\mathbb{T})}\lesssim\|\Phi_{i}\|_{C^{1+\alpha}(\mathbb{T})}+\|\Phi_{j}\|_{C^{1+\alpha}(\mathbb{T})}.$$
	We conclude that,
	$$\| D_{f_{j}}\mathcal{I}_{j}(\lambda,f_{1},f_{2})h_{j}\|_{C^{\alpha}(\mathbb{T})}\lesssim\| h_{j}\|_{C^{1+\alpha}(\mathbb{T})},$$
	which means that $D_{f_{j}}\mathcal{I}_{j}(\lambda,b,f_{1},f_{2})\in\mathcal{L}(C^{1+\alpha}(\mathbb{T}),C^{\alpha}(\mathbb{T})).$\\
	\ding{226} Concerning the other differentiation, we have  
	\begin{align}\label{DfiNj general}
		D_{f_{i}}\mathcal{I}_{j}&(\lambda,b,f_{1},f_{2})h_{i}(w)  =  \displaystyle(-1)^{j-1}\mbox{Im}\left\lbrace\overline{w}\overline{\Phi_{j}'(w)}\fint_{\mathbb{T}}h_{i}'(\tau)K_{0}\left(\lambda|\Phi_{j}(w)-\Phi_{i}(\tau)|\right)d\tau\right.\nonumber\\
		&  \displaystyle-\frac{\lambda}{2}\overline{w}\overline{\Phi_{j}'(w)}\fint_{\mathbb{T}}h_{i}(\tau)\Phi_{i}'(\tau)K_{0}'\left(\lambda|\Phi_{j}(w)-\Phi_{i}(\tau)|\right)\frac{\Phi_{j}(\overline{w})-\Phi_{i}(\overline{\tau})}{|\Phi_{j}(w)-\Phi_{i}(\tau)|}d\tau\nonumber\\
		&  \displaystyle\left.-\frac{\lambda}{2}\overline{w}\overline{\Phi_{j}'(w)}\fint_{\mathbb{T}}\overline{h_{i}(\tau)}\Phi_{i}'(\tau)K_{0}'\left(\lambda|\Phi_{j}(w)-\Phi_{i}(\tau)|\right)\frac{\Phi_{j}(w)-\Phi_{i}(\tau)}{|\Phi_{j}(w)-\Phi_{i}(\tau)|}d\tau\right\rbrace\nonumber\\
		& :=  (-1)^{j-1}\mbox{Im}\Big\{\overline{w}\overline{\Phi_{j}'(w)}\big[C(\lambda,\Phi_{i},\Phi_{j})(h_{i})(w)+D(\lambda,\Phi_{i},\Phi_{j})(h_{i})(w)+E(\lambda,\Phi_{i},\Phi_{j})(h_{i})(w)\big]\Big\}.
	\end{align}
	Using the algebra structure of $C^{\alpha}(\mathbb{T})$, we obtain 
	$$\| D_{f_{i}}\mathcal{I}_{j}(\lambda,b,f_{1},f_{2})h_{i}\|_{C^{\alpha}(\mathbb{T})}\lesssim\| C(\lambda,\Phi_{i},\Phi_{j})h_{i}\|_{C^{\alpha}(\mathbb{T})}+\| D(\lambda,\Phi_{i},\Phi_{j})h_{i}\|_{C^{\alpha}(\mathbb{T})}+\| E(\lambda,\Phi_{i},\Phi_{j})h_{i}\|_{C^{\alpha}(\mathbb{T})}.$$
	From \eqref{estimate for mathcalT}, we find 
	$$\| C(\lambda,\Phi_{i},\Phi_{j})h_{i}\|_{C^{\alpha}(\mathbb{T})}\lesssim\| h_{i}'\|_{L^{\infty}(\mathbb{T})}\leqslant\| h_{i}\|_{C^{1+\alpha}(\mathbb{T})}.$$
	In the same way as for $A(\lambda,\Phi_{i},\Phi_{j})$, we infer 
	$$\| D(\lambda,\Phi_{i},\Phi_{j})h_{i}\|_{C^{\alpha}(\mathbb{T})}+\| E(\lambda,\Phi_{i},\Phi_{j})h_{i}\|_{C^{\alpha}(\mathbb{T})}\lesssim\| h_{i}\|_{L^{\infty}(\mathbb{T})}\leqslant\| h_{i}\|_{C^{1+\alpha}(\mathbb{T})}.$$
	Gathering the foregoing computations leads to
	$$\| D_{f_{i}}\mathcal{I}_{j}(\lambda,b,f_{1},f_{2})h_{i}\|_{C^{\alpha}(\mathbb{T})}\lesssim\| h_{i}\|_{C^{1+\alpha}(\mathbb{T})},$$
	that is, $D_{f_{i}}\mathcal{I}_{j}(\lambda,b,f_{1},f_{2})\in\mathcal{L}(C^{1+\alpha}(\mathbb{T}),C^{\alpha}(\mathbb{T})).$\\
	\ding{226} The last thing to check is that the convergence in \eqref{cv DNj} occurs in the strong topology of $C^{\alpha}(\mathbb{T}).$ Since there are many terms involved, we shall select the more complicated one and study it. The other terms can be treated in a similar way, up to slight modifications. Let us focus on the first term of the right-hand side of \eqref{DfjNj general}. We shall prove,
	$$\lim_{t\rightarrow 0}S(\lambda,\Phi_{i},\Phi_{i}+th_{j})-S(\lambda,\Phi_{i},\Phi_{j})=0\quad\mbox{in}\quad C^{\alpha}(\mathbb{T}).$$
	For more convenience, we use the following notation
	$$T_{ij}(\lambda,t,w):=S(\lambda,\Phi_{i},\Phi_{i}+th_{j})(w)-S(\lambda,\Phi_{i},\Phi_{j})(w).$$
	Consider $t>0$ such that $t\| h_{j}\|_{L^{\infty}(\mathbb{T})}<r.$ According to \eqref{def S}, we get 
	\begin{align*}
		T_{ij}(\lambda,t,w) & =  \displaystyle\fint_{\mathbb{T}}\Phi_{i}'(\tau)\left(K_{0}\left(\lambda\left|\Phi_{j}(w)-\Phi_{i}(\tau)+th_{j}(w)\right|\right)-K_{0}\left(\lambda\left|\Phi_{j}(w)-\Phi_{i}(\tau)\right|\right)\right)d\tau\\
		& :=  \displaystyle\fint_{\mathbb{T}}\Phi_{i}'(\tau)\mathbb{K}(\lambda,t,w,\tau)d\tau.
	\end{align*}
	Applying mean value Theorem and left triangle inequality, we obtain
	$$\left|\mathbb{K}(\lambda,t,w,\tau)\right|\lesssim t\| h_{j}\|_{L^{\infty}(\mathbb{T})}.$$
	Consequently,
	$$\left|T_{ij}(\lambda,t,w)\right|\lesssim t\| h_{j}\|_{L^{\infty}(\mathbb{T})}.$$
	This implies that 
	$$\lim_{t\rightarrow 0}\| T_{ij}(\lambda,t,\cdot)\|_{L^{\infty}(\mathbb{T})}=0.$$
	Let us now consider $w_{1}\neq w_{2}\in\mathbb{T}.$ In view of the mean value Theorem, one obtains the following estimate
	\begin{align}\label{lip Tij}
		\left|T_{ij}(\lambda,t,w_{1})-T_{ij}(\lambda,t,w_{2})\right| & \lesssim  \displaystyle\int_{\mathbb{T}}\left|\mathbb{K}(\lambda,t,w_{1},\tau)-\mathbb{K}(\lambda,t,w_{2},\tau)\right||d\tau|\nonumber\\
		& \lesssim  \displaystyle|w_{1}-w_{2}|\int_{\mathbb{T}}\sup_{w\in\mathbb{T}}\left|\partial_{w}\mathbb{K}(\lambda,t,w,\tau)\right||d\tau|.
	\end{align}
	Now remark that we can write
	$$\mathbb{K}(\lambda,t,w,\tau)=\int_{0}^{t}\partial_{s}g(\lambda,s,w,\tau)ds\quad\mbox{with}\quad g(\lambda,t,w,\tau):=K_{0}\left(\lambda\left|\Phi_{j}(w)-\Phi_{i}(\tau)+\tau h_{j}(w)\right|\right).$$
	According to \eqref{formula derivative conjugate}, one obtains
	\begin{align*}
		\partial_{w}&g(\lambda,t,w,\tau)=\frac{\lambda}{2} K_{0}'\left(\lambda\left|\Phi_{j}(w)-\Phi_{i}(\tau)+th_{j}(w)\right|\right)\\
		&\times\frac{\left(\Phi_{j}'(w)+th_{j}'(w)\right)\left(\overline{\Phi_{j}(w)}-\overline{\Phi_{i}(\tau)}+t\overline{h_{j}(w)}\right)-\overline{w}^{2}\left(\overline{\Phi_{j}'(w)}+t\overline{h_{j}'(w)}\right)\left(\Phi_{j}(w)-\Phi_{i}(\tau)+th_{j}(w)\right)}{\left|\Phi_{j}(w)-\Phi_{i}(\tau)+th_{j}(w)\right|}.
	\end{align*}
	After straightforward computations, we obtain for $s\in[0,t]$,
	$$\left|\partial_{s}\partial_{w}g(\lambda,s,w,\tau)\right|\lesssim 1.$$
	As a consequence, we infer
	$$\left|\partial_{w}\mathbb{K}(\lambda,t,w,\tau)\right|\lesssim |t|.$$
	Coming back to \eqref{lip Tij} and using the fact that $\alpha\in(0,1)$, we conclude
	$$\left|T_{ij}(\lambda,t,w_{1})-T_{ij}(\lambda,t,w_{2})\right|\lesssim |t||w_{1}-w_{2}|\lesssim |t||w_{1}-w_{2}|^{\alpha}.$$
	Therefore,
	$$\lim_{t\rightarrow 0}\| T_{ij}(t,\cdot)\|_{C^{\alpha}(\mathbb{T})}=0.$$
	The second step is now achieved.\\
	
	\noindent $\blacktriangleright$ \textbf{Step 3 : Show that the Gâteaux derivatives of $G(\lambda,b,\cdot,\cdot,\cdot)$ are continuous :}\\
	Now we investigate for the continuity of the Gâteaux derivatives seen as operators from the neighborhood $B_{r}^{1+\alpha}\times B_{r}^{1+\alpha}$ into the Banach space $\mathcal{L}\left(X_{1}^{1+\alpha},Y_{1}^{\alpha}\right).$ Using the algebra structure of $C^{\alpha}(\mathbb{T})$, we deduce from \eqref{DfiNj general} and \eqref{DfjNj general} that we only have to study the continuity of the terms $S(\lambda,\Phi_{i},\Phi_{j})$, $A(\lambda,\Phi_{i},\Phi_{j})$, $B(\lambda,\Phi_{i},\Phi_{j})$, $C(\lambda,\Phi_{i},\Phi_{j})h_{i}$, $D(\lambda,\Phi_{i},\Phi_{j})h_{i}$ and $E(\lambda,\Phi_{i},\Phi_{j})h_{i}.$ As before, we shall focus on the term $S(\lambda,\Phi_{i},\Phi_{j})$ for $i\neq j$ and remark that the other terms are similar. We denote
	$$\Phi_{1}:=\textnormal{Id}+f_{1},\quad\Psi_{1}:=\textnormal{Id}+g_{1},\quad\Phi_{2}:=b\textnormal{Id}+f_{2},\quad\Psi_{2}:=b\textnormal{Id}+g_{2},$$
	with $(f_{1},f_{2})\in B_{r}^{1+\alpha}\times B_{r}^{1+\alpha}$ and $(g_{1},g_{2})\in B_{r}^{1+\alpha}\times B_{r}^{1+\alpha}.$ Let us show that 
	$$\| S(\lambda,\Phi_{i},\Phi_{j})-S(\lambda,\Psi_{i},\Psi_{j})\|_{C^{\alpha}(\mathbb{T})}\lesssim\| f_{1}-g_{1}\|_{C^{1+\alpha}(\mathbb{T})}+\| f_{2}-g_{2}\|_{C^{1+\alpha}(\mathbb{T})}.$$
	According to \eqref{def S}, we get 
	\begin{align*}
		S(\lambda,\Phi_{i},\Phi_{j})(w)-S(\lambda,\Psi_{i},\Psi_{j})(w) & =  \displaystyle\fint_{\mathbb{T}}\left[\Phi_{i}'(\tau)K_{0}\left(\lambda\left|\Phi_{j}(w)-\Phi_{i}(\tau)\right|\right)-\Psi_{i}'(\tau)K_{0}\left(\lambda\left|\Psi_{j}(w)-\Psi_{i}(\tau)\right|\right)\right]d\tau\\
		& :=  \displaystyle\fint_{\mathbb{T}}\Psi_{i}'(\tau)\mathbb{K}_{2}(\lambda,w,\tau)d\tau+\fint_{\mathbb{T}}\left(\Phi_{i}'(\tau)-\Psi_{i}'(\tau)\right)K_{0}\left(\lambda\left|\Phi_{j}(w)-\Phi_{i}(\tau)\right|\right)d\tau,
	\end{align*}
	where
	$$\mathbb{K}_{2}(\lambda,w,\tau):=K_{0}\left(\lambda\left|\Phi_{j}(w)-\Phi_{i}(\tau)\right|\right)-K_{0}\left(\lambda\left|\Psi_{j}(w)-\Psi_{i}(\tau)\right|\right).$$
	We have directly
	$$\Big\|\fint_{\mathbb{T}}\left(\Phi_{i}'(\tau)-\Psi_{i}'(\tau)\right)K_{0}\left(\lambda\left|\Phi_{j}(\cdot)-\Phi_{i}(\tau)\right|\right)d\tau\Big\|_{C^{\alpha}(\mathbb{T})}\lesssim\| f_{i}'-g_{i}'\|_{L^{\infty}(\mathbb{T})}\leqslant\| f_{i}-g_{i}\|_{C^{1+\alpha}(\mathbb{T})}.$$
	Now set
	$$L_{i}(\lambda,w):=\fint_{\mathbb{T}}\mathbb{K}_{2}(\lambda,w,\tau)\Psi_{i}'(\tau)d\tau,$$
	By a new use of the mean value Theorem and left triangle inequality, we obtain 
	\begin{align*}
		|\mathbb{K}_{2}(\lambda,w,\tau)| & \lesssim  \big|\left|\Phi_{j}(w)-\Phi_{i}(\tau)\right|-\left|\Psi_{j}(w)-\Psi_{i}(\tau)\right|\big|\\
		& \leqslant  \left|\Phi_{j}(w)-\Psi_{j}(w)\right|+\left|\Phi_{i}(\tau)-\Psi_{i}(\tau)\right|\\
		& \leqslant  \|\Psi_{j}-\Phi_{j}\|_{L^{\infty}(\mathbb{T})}+\|\Psi_{i}-\Phi_{i}\|_{L^{\infty}(\mathbb{T})}.
	\end{align*}
	Hence, we deduce 
	\begin{align*}
		\| L_{i}(\lambda,\cdot)\|_{L^{\infty}(\mathbb{T})} & \lesssim  \|\Psi_{i}'\|_{L^{\infty}(\mathbb{T})}\left(\|\Psi_{j}-\Phi_{j}\|_{L^{\infty}(\mathbb{T})}+\|\Psi_{i}-\Phi_{i}\|_{L^{\infty}(\mathbb{T})}\right)\\
		& \lesssim  \| f_{j}-g_{j}\|_{C^{1+\alpha}(\mathbb{T})}+\| f_{i}-g_{i}\|_{C^{1+\alpha}(\mathbb{T})}.
	\end{align*}
	Take $w_{1}\neq w_{2}\in\mathbb{T}.$ Applying the mean value Theorem yields
	$$\left|L_{i}(\lambda,w_{1})-L_{i}(\lambda,w_{2})\right|\lesssim|w_{1}-w_{2}|\fint_{\mathbb{T}}\sup_{w\in\mathbb{T}}\left|\partial_{w}\mathbb{K}_{2}(\lambda,w,\tau)\right||d\tau|.$$
	By \eqref{formula derivative conjugate}, we have 
	$$\partial_{w}\mathbb{K}_{2}(\lambda,w,\tau)=\frac{\lambda}{2}\left(\overline{\mathcal{J}(\lambda,w,\tau)}-\overline{w}^{2}\mathcal{J}(\lambda,w,\tau)\right),$$
	where
	$$\mathcal{J}(\lambda,w,\tau):=\overline{\Phi_{j}'(w)}(\Phi_{j}(w)-\Phi_{i}(\tau))K_{0}'\left(\lambda|\Phi_{j}(w)-\Phi_{i}(\tau)|\right)-\overline{\Psi_{j}'(w)}(\Psi_{j}(w)-\Psi_{i}(\tau))K_{0}'\left(\lambda|\Psi_{j}(w)-\Psi_{i}(\tau)|\right).$$
	Notice that it can be written in the following form
	$$\mathcal{J}(\lambda,w,\tau)=\mathcal{J}_{1}(\lambda,w,\tau)+\mathcal{J}_{2}(\lambda,w,\tau)+\mathcal{J}_{3}(\lambda,w,\tau),$$
	with 
	\begin{align*}
		\mathcal{J}_{1}(\lambda,w,\tau)&:=\overline{\Phi_{j}'(w)}\left[(\Phi_{j}-\Psi_{j})(w)-(\Phi_{i}-\Psi_{i})(\tau)\right]K_{0}'\left(\lambda|\Phi_{j}(w)-\Phi_{i}(\tau)|\right),\\
		\mathcal{J}_{2}(\lambda,w,\tau)&:=\left[\overline{\Phi_{j}'(w)}-\overline{\Psi_{j}'(w)}\right]\left[\Psi_{j}(w)-\Psi_{i}(\tau)\right]K_{0}'\left(\lambda|\Psi_{j}(w)-\Psi_{i}(\tau)|\right),\\
		\mathcal{J}_{3}(\lambda,w,\tau)&:=\overline{\Phi_{j}'(w)}\left[\Psi_{j}(w)-\Psi_{i}(\tau)\right]\left[K_{0}'\left(\lambda|\Phi_{j}(w)-\Phi_{i}(\tau)|\right)-K_{0}'\left(\lambda|\Psi_{j}(w)-\Psi_{i}(\tau)|\right)\right].
	\end{align*}
	By the same techniques as already used above, we get 
	$$\|\partial_{w}\mathbb{K}_{2}(\lambda,\cdot,\tau)\|_{L^{\infty}(\mathbb{T})}\lesssim\| f_{j}-g_{j}\|_{C^{1+\alpha}(\mathbb{T})}+\| f_{i}-g_{i}\|_{C^{1+\alpha}(\mathbb{T})}.$$
	We deduce that
	$$\| S(\lambda,\Phi_{i},\Phi_{j})-S(\lambda,\Psi_{i},\Psi_{j})\|_{C^{\alpha}(\mathbb{T})}\lesssim\| f_{j}-g_{j}\|_{C^{1+\alpha}(\mathbb{T})}+\| f_{i}-g_{i}\|_{C^{1+\alpha}(\mathbb{T})}.$$
	\textbf{(ii)} Looking at Proposition \ref{proposition regularity of the functional}, it is sufficient to prove the preservation of the $\mathbf{m}$-fold symmetry. Let $r$ be as in Proposition \ref{proposition regularity of the functional}. Let $(f_{1},f_{2})\in B_{r,\mathbf{m}}^{1+\alpha}\times B_{r,\mathbf{m}}^{1+\alpha}.$ Let $\Phi_{1}$ and $\Phi_{2}$ be the associated conformal maps 
	$$\Phi_{1}(z)=z+\sum_{n=0}^{\infty}\frac{a_{n}}{z^{\mathbf{m}n-1}}\quad\mbox{ and }\quad\Phi_{2}(z)=bz+\sum_{n=0}^{\infty}\frac{b_{n}}{z^{\mathbf{m}n-1}}.$$
	One easily obtains 
	$$\forall j\in\{1,2\},\quad\forall w\in\mathbb{T},\quad\Phi_{j}\left(e^{\frac{2\ii\pi}{\mathbf{m}}}w\right)=e^{\frac{2\ii\pi}{\mathbf{m}}}\Phi_{j}(w)\quad\mbox{ and }\quad\Phi_{j}'\left(e^{\frac{2\ii\pi}{\mathbf{m}}}w\right)=\Phi_{j}(w).$$
	Hence, by using the change of variables $\tau\mapsto e^{\frac{2\ii\pi}{\mathbf{m}}}\tau$, we have for all $(i,j)\in\{1,2\}^{2}$ and for all $w\in\mathbb{T}$,
	\begin{align*}
		S(\lambda,\Phi_{i},\Phi_{j})\left(e^{\frac{2\ii\pi}{\mathbf{m}}}w\right) & =  \displaystyle\fint_{\mathbb{T}}\Phi_{i}'(\tau)K_{0}\left(\lambda\left|\Phi_{j}\left(e^{\frac{2\ii\pi}{\mathbf{m}}}w\right)-\Phi_{i}\left(\tau\right)\right|\right)d\tau\\
		& =  \displaystyle e^{\frac{2\ii\pi}{\mathbf{m}}}\fint_{\mathbb{T}}\Phi_{i}'\left(e^{\frac{2\ii\pi}{\mathbf{m}}}\tau\right)K_{0}\left(\lambda\left|\Phi_{j}\left(e^{\frac{2\ii\pi}{\mathbf{m}}}w\right)-\Phi_{i}\left(e^{\frac{2\ii\pi}{\mathbf{m}}}\tau\right)\right|\right)d\tau\\
		& =  \displaystyle e^{\frac{2\ii\pi}{\mathbf{m}}}\fint_{\mathbb{T}}\Phi_{i}'(\tau)K_{0}\left(\lambda\left|\Phi_{j}\left(w\right)-\Phi_{i}\left(\tau\right)\right|\right)d\tau\\
		& =  \displaystyle e^{\frac{2\ii\pi}{\mathbf{m}}}S(\lambda,\Phi_{i},\Phi_{j})(w).
	\end{align*}
	By definition \eqref{definition of Fj} of $G_{j}$, this immediately implies that
	$$\forall j\in\{1,2\},\quad\forall w\in\mathbb{T},\quad G_{j}\left(\lambda,b,\Omega,f_{1},f_{2}\right)\left(e^{\frac{2\ii\pi}{\mathbf{m}}}w\right)=G_{j}\left(\lambda,b,\Omega,f_{1},f_{2}\right)\left(w\right).$$
	So $$G(\lambda,b,\cdot,\cdot,\cdot):\mathbb{R}\times B_{r,\mathbf{m}}^{1+\alpha}\times B_{r,\mathbf{m}}^{1+\alpha}\rightarrow Y_{\mathbf{m}}^{\alpha}.$$
	\textbf{(iii)} Fix $j\in\{1,2\}.$ By \eqref{split GON} and \eqref{DOj}, we have for $f_{j}\in B_{r}^{1+\alpha}$ and $h_{j}\in C^{1+\alpha}(\mathbb{T}),$ 
	\begin{align*}
		\partial_{\Omega}D_{f_{j}}G_{j}(\lambda,b,\Omega,f_{j})(h_{j})(w)&=\partial_{\Omega}D_{f_{j}}\mathcal{S}_{j}(\lambda,b,\Omega,f_{j})(h_{j})(w)\\
		&=\mbox{Im}\left\lbrace h_{j}(w)\overline{w}\overline{\Phi_{j}'(w)}+\Phi_{j}(w)\overline{w}\overline{h_{j}'(w)}\right\rbrace.
	\end{align*}
	As a consequence, we deduce that for $(f_{j},g_{j})\in(B_{r}^{1+\alpha})^{2}$ and $h_{j}\in C^{1+\alpha}(\mathbb{T})$,  
	$$\Big\|\partial_{\Omega}D_{f_{j}}G_{j}(\lambda,b,\Omega,f_{j})(h_{j})-\partial_{\Omega}D_{f_{j}}G_{j}(\lambda,b,\Omega,g_{j})(h_{j})\Big\|_{C^{\alpha}(\mathbb{T})}\lesssim\| f_{j}-g_{j}\|_{C^{1+\alpha}(\mathbb{T})}\| h_{j}\|_{C^{1+\alpha}(\mathbb{T})}.$$
	This proves the continuity of $\partial_{\Omega}DG(\lambda,b,\cdot,\cdot,\cdot):\mathbb{R}\times B_{r}^{1+\alpha}\times B_{r}^{1+\alpha}\rightarrow\mathcal{L}(X^{1+\alpha},Y^{\alpha})$ and achieves the proof of Proposition \ref{proposition regularity of the functional}.
\end{proof}
\section{Crandall-Rabinowitz's Theorem}\label{appendix CR}
Now, we recall the classical Crandall-Rabinowitz's Theorem. This result was first proved in \cite{CR71} and it  is one of the most common theorems appearing in the bifurcation theory. A convenient reference in the subject is \cite{K11}. We briefly explain the core of local bifurcation theory.\\

Consider a function $F:\mathbb{R}\times X\rightarrow Y$ with $X$ and $Y$ two Banach spaces. Assume that for all $\Omega$ in a non-empty interval $I$ we have $F(\Omega,0)=0.$ This provides a line of solutions
$$\big\{(\Omega,0),\quad\Omega\in I\big\}.$$
Now take some $(\Omega_{0},0)$ with $\Omega_{0}\in I.$ The implicit function Theorem explains that if $DF(\Omega_{0},0)$ is invertible, then the line $\{(\Omega,0),|\Omega-\Omega_{0}|\leqslant \varepsilon\}$ is the only curve of solutions close to $(\Omega_{0},0)$, i.e. for $\varepsilon$ small enough. (Local) bifurcation theory is the study of situations where this is not true, that is, close to $(\Omega_{0},0)$ there exists (at least) another line of solutions. In this case, we say that $(\Omega_{0},0)$ is a bifurcation point. Crandall-Rabinowitz's Theorem gives sufficient conditions to construct a bifurcation curve and states as follows.
\begin{theo}[Crandall-Rabinowitz]\label{Crandall-Rabinowitz theorem}
	Let $X$ and $Y$ be two banach spaces. Let $V$ be a neighborhood of $0$ in $X$ and let
	$$F:\begin{array}[t]{rcl}
		\mathbb{R}\times V & \rightarrow  & Y\\
		(\Omega,x) & \mapsto & F(\Omega,x)
	\end{array}$$
	be a function of classe $C^{1}$ with the following properties 
	\begin{enumerate}[label=(\roman*)]
		\item (Trivial solution) $\forall\,\Omega\in\mathbb{R},F(\Omega,0)=0.$
		\item (Regularity) $F_{\Omega}$, $F_{x}$ and $F_{\Omega x}$ exist and are continuous.
		\item (Fredholm property) $\ker\left(\partial_{x}F(0,0)\right)=\langle x_{0}\rangle$ and $Y/R\left(\partial_{x}F(0,0)\right)$ are one dimensional and $R\left(\partial_{x}F(0,0)\right)$ is closed in $Y.$
		\item (Transversality assumption) $\partial_{\Omega}\partial_{x}F(0,0)x_{0}\not\in R\left(\partial_{x}F(0,0)\right).$
	\end{enumerate}
	If $\chi$ is any complement of $\ker\left(\partial_{x}F(0,0)\right)$ in $X$, then there exist a neighborhood $U$ of $(0,0)$, an interval $(-a,a)$ ($a>0$) and continuous functions
	$$\psi:(-a,a)\rightarrow \mathbb{R}\quad\textnormal{and}\quad\phi:(-a,a)\rightarrow\chi$$
	such that $\psi(0)=0,$ $\phi(0)=0$ and 
	$$\Big\{(\Omega,x)\in U\quad\textnormal{s.t.}\quad F(\Omega,x)=0\Big\}=\Big\{\big(\psi(s),sx_{0}+s\phi(s)\big)\quad\textnormal{s.t.}\quad|s|<a\Big\}\cup\Big\{(\Omega,0)\in U\Big\}.$$
\end{theo}

\end{document}